\documentclass[a4paper,10pt]{article}

\usepackage[english]{babel}
\usepackage{amsmath,amsfonts,amssymb,amsthm}
\usepackage{mathrsfs}
\usepackage{bbm}
\usepackage{indentfirst}

\newtheorem{theorem}{Theorem}[section]

\newtheorem{lemma}{Lemma}[section]
\newtheorem{proposition}{Proposition}[section]
\newtheorem{corollary}{Corollary}[section]

\theoremstyle{definition}
\newtheorem{remark}{Remark}[section]

\numberwithin{equation}{section}

\newcommand\blfootnote[1]{\begingroup\renewcommand\thefootnote{}\footnote{#1}\addtocounter{footnote}{-1}\endgroup}

\begin{document}

\title{
{\bf\Large Existence of positive solutions in the superlinear case via coincidence degree: the Neumann and the periodic boundary value problems }\footnote{Work performed under the auspices of the
Grup\-po Na\-zio\-na\-le per l'Anali\-si Ma\-te\-ma\-ti\-ca, la Pro\-ba\-bi\-li\-t\`{a} e le lo\-ro
Appli\-ca\-zio\-ni (GNAMPA) of the Isti\-tu\-to Na\-zio\-na\-le di Al\-ta Ma\-te\-ma\-ti\-ca (INdAM).}
\vspace{1mm}
}

\author{
{\bf\large Guglielmo Feltrin}
\vspace{1mm}\\
{\it\small SISSA - International School for Advanced Studies}\\
{\it\small via Bonomea 265}, {\it\small 34136 Trieste, Italy}\\
{\it\small e-mail: guglielmo.feltrin@sissa.it}\vspace{1mm}\\
\vspace{1mm}\\
{\bf\large Fabio Zanolin}
\vspace{1mm}\\
{\it\small Department of Mathematics and Computer Science, University of Udine}\\
{\it\small via delle Scienze 206},
{\it\small 33100 Udine, Italy}\\
{\it\small e-mail: fabio.zanolin@uniud.it}}

\date{}

\maketitle

\vspace{-2mm}

\begin{abstract}
\noindent
We prove the existence of positive periodic solutions for the second order nonlinear equation
$u'' + a(x) g(u) = 0$,
where $g(u)$ has superlinear growth at zero and at infinity. The weight function $a(x)$ is
allowed to change its sign. Necessary and sufficient conditions for the existence of nontrivial solutions are obtained.
The proof is based on Mawhin's coincidence degree and applies also to Neumann boundary conditions.
Applications are given to the search of positive solutions for a nonlinear PDE in annular domains
and for a periodic problem associated to a non-Hamiltonian equation.
\end{abstract}

\section{Introduction}\label{section1}

Let ${\mathbb{R}}^{+}:=\mathopen{[}0,+\infty\mathclose{[}$ denote the set of non-negative real numbers and let
$g\colon {\mathbb{R}}^{+} \to {\mathbb{R}}^{+}$ be a continuous function such that
\begin{equation*}
g(0) = 0, \qquad g(s) > 0 \quad  \text{for } \; s > 0.
\leqno{\hspace*{2.2pt}(g_{1})}
\end{equation*}
In the present paper we study the second order nonlinear boundary value problem
\begin{equation*}
\begin{cases}
\, u'' + a(x) g(u) = 0, \quad 0 < x < T, \\
\, {\mathscr{B}}(u,u') = \underline{0}.
\end{cases}
\leqno{\hspace*{2.2pt}({\mathscr{P}})}
\end{equation*}
As linear boundary operator we take
\begin{equation*}
{\mathscr{B}}(u,u') = (u'(0), u'(T))
\end{equation*}
or
\begin{equation*}
{\mathscr{B}}(u,u') = (u(T) - u(0), u'(T) - u'(0)),
\end{equation*}
so that we consider the Neumann and the periodic boundary value problems.
The weight $a(x)$ is a Lebesgue integrable function defined on $\mathopen{[}0,T\mathclose{]}$.
A \textit{solution} of $({\mathscr{P}})$ is a continuously differentiable function $u \colon \mathopen{[}0,T\mathclose{]}\to {\mathbb{R}}$
such that its derivative $u'(x)$ is absolutely continuous and $u(x)$ satisfies $({\mathscr{P}})$
for a.e.~$x\in\mathopen{[}0,T\mathclose{]}$.
We look for \textit{positive solutions} of $({\mathscr{P}})$,
that is solutions $u$ such that $u(x)>0$ for every $x\in \mathopen{[}0,T\mathclose{]}$.
Of course, if $a(x)$ is continuous, then $u(x)$ is a classical solution of class $\mathcal{C}^{2}$.
In relation to the Neumann and the periodic boundary value problems,
assumption $(g_{1})$, which requires that $g(s)$ never vanishes on ${\mathbb{R}}^{+}_{0}:=\mathopen{]}0,+\infty\mathclose{[}$,
is essential to guarantee that the positive solutions we find are not constant.

Boundary value problems associated to second order differential equations of the form
\begin{equation*}
u'' + a(x)g(u) = 0
\end{equation*}
arise from many different areas of research, in particular they play a relevant role in
the study of stationary solutions of reaction-diffusion equations.
In this context, the search of positive solutions is of great interest
in many applications to population dynamics and mathematical ecology
(see~\cite{BaPoTe-1987,BaPoTe-1988,GoReLoGo-2000} and also \cite{Ac-2009} for a recent survey on that topic).

If $u(x)$ is any positive solution to the BVP $({\mathscr{P}})$, then an integration on $\mathopen{[}0,T\mathclose{]}$ yields
$\int_{0}^{T} a(x)g(u(x)) ~\!dx = 0$ and this fact, in connection with $(g_{1})$, implies that the weight function
$a(x)$ (if not identically zero) must change its sign. A second relation can be derived when $g(s)$
is continuously differentiable on ${\mathbb{R}}^{+}_{0}$.
Indeed, dividing the equation by $g(u(x))$ and integrating by parts, we obtain
\begin{equation*}
\int_{0}^{T} {g'(u(x))}\biggl(\dfrac{u'(x)}{g(u(x))}\biggr)^{\!2}~\!dx = -\int_{0}^{T} a(x)~\!dx
\end{equation*}
(cf.~\cite{BaPoTe-1988,BoZa-2012}). From this relation, if $g'(s)> 0$ on ${\mathbb{R}}^{+}_{0}$, we find that
a necessary condition for the existence of positive solutions is
\begin{equation*}
\int_{0}^{T} a(x)~\!dx < 0.
\end{equation*}

The above remarks suggest that, if we want to find nontrivial positive solutions for $({\mathscr{P}})$
with nonlinearities which include as a particular possibility the case of $g(s)$ strictly monotone,
we have to study problem $({\mathscr{P}})$ considering sign-indefinite weight functions with negative mean value
on $\mathopen{[}0,T\mathclose{]}$.

In this paper we study the case of nonlinearities $g(s)$ which have a superlinear growth at zero and at
infinity. Boundary value problems of this form are usually named of \textit{superlinear indefinite} type (cf.~\cite{BeCaDoNi-1994}).
In the past twenty years a great deal of existence and multiplicity results have been reached in this
context, mainly with respect to Dirichlet boundary conditions. The starting point for our investigation is the following result
obtained in \cite{GaHaZa-2003} for the two-point boundary value problem.

\begin{theorem}\label{th-1.1}
Assume that $\{x\in \mathopen{[}0,T\mathclose{]} \colon a(x) > 0\} = \bigcup_{i=1}^{k} J_{i} \neq \emptyset$, where the sets $J_{i}$
are pairwise disjoint intervals, and suppose that
\begin{equation}\label{eq-1.1}
\limsup_{s\to 0^{+}} \dfrac{g(s)}{s} < \lambda_{0} \quad \text{ and } \quad
\liminf_{s\to +\infty} \dfrac{g(s)}{s} > \max_{i=1,\ldots,k}\lambda_{1}^{i},
\end{equation}
where $\lambda_{0}$ is the first eigenvalue of the eigenvalue problem
\begin{equation}\label{eq-1.2}
\varphi'' + \lambda a^{+}(x) \varphi = 0, \quad \varphi(0) = \varphi(T) = 0,
\end{equation}
and $\lambda_{1}^{i}$ ($i=1,\ldots,k$) is the first eigenvalue of the eigenvalue problem
\begin{equation*}
\varphi'' + \lambda a(x) \varphi = 0, \quad \varphi|_{\partial J_{i}} = 0.
\end{equation*}
Then there is at least one solution to
\begin{equation}\label{eq-1.3}
u'' + a(x) g(u) = 0, \quad u(0) = u(T) = 0,
\end{equation}
which is positive on $\mathopen{]}0,T\mathclose{[}$.
\end{theorem}

An immediate corollary of this result can be attained by assuming, instead of \eqref{eq-1.1}, that
\begin{equation}\label{eq-1.4}
\lim_{s\to 0^{+}} \dfrac{g(s)}{s} =0 \quad \text{ and } \quad
\lim_{s\to +\infty} \dfrac{g(s)}{s} = +\infty,
\end{equation}
which means that $g(s)$ goes to zero (respectively, to infinity) faster than linear.
As far as we know, there is no counterpart of this theorem for the Neumann or the periodic problem $({\mathscr{P}})$.
Indeed, if we try to mimic the above result for our boundary conditions, we have that \eqref{eq-1.2} reads as
\begin{equation*}
\varphi'' + \lambda a^{+}(x) \varphi = 0,\quad {\mathscr{B}}(\varphi,\varphi') = \underline{0},
\end{equation*}
which has $\lambda_{0} = 0$ as first eigenvalue. Hence the assumption on the $\limsup$ in \eqref{eq-1.1}
becomes inconsistent in view of condition $(g_{1})$.

In the present paper we propose a version of Theorem~\ref{th-1.1} for problem $({\mathscr{P}})$ which can be stated as follows.

\begin{theorem}\label{th-1.2}
Assume that $\{x\in \mathopen{[}0,T\mathclose{]} \colon a(x) > 0\} = \bigcup_{i=1}^{k} J_{i} \neq \emptyset$, where the sets $J_{i}$
are pairwise disjoint intervals, and suppose that $g(s)$ is a smooth function satisfying $(g_{1})$ and such that
\begin{equation*}
\lim_{s\to 0^{+}} \dfrac{g(s)}{s} = 0 \quad \text{ and } \quad
g_{\infty}:=\liminf_{s\to +\infty} \dfrac{g(s)}{s} > \max_{i=1,\ldots,k}\lambda_{1}^{i}
\end{equation*}
(with $\lambda_{1}^{i}$ as above).
Then there is at least one positive solution of $({\mathscr{P}})$ provided that
$\int_{0}^{T} a(x) ~\!dx < 0$.
\end{theorem}

To be more precise, the assumption that we actually need on the
weight function $a(x)$ is slightly more general than the one in
Theorem~\ref{th-1.2} (see condition $(a_{1})$ in
Section~\ref{section3}). The condition on the average of $a(x)$
is new with respect to the Dirichlet
case, but, in relation to the Neumann and the periodic boundary
conditions, it becomes necessary in a way (at least for
nonlinearities with $g'(s) > 0$). We underline that the hypothesis
of smoothness for $g(s)$ considered in Theorem~\ref{th-1.2} has
been chosen only to simplify the presentation and it can be
improved by requiring $g(s)$ continuously differentiable on a
right neighborhood of $s=0$ (as in Theorem~\ref{MainTheorem2}) or
only continuous but regularly oscillating at zero (as in
Theorem~\ref{MainTheorem}).

We notice that Theorem~\ref{th-1.2} allows to study both the case of nonlinearities which are superlinear at zero and at infinity, as in \eqref{eq-1.4},
and that of nonlinearities which are only superlinear at zero and with linear growth at infinity. With this respect, the following result holds.

\begin{theorem}\label{th-1.3}
Assume there exists an interval $J\subseteq \mathopen{[}0,T\mathclose{]}$ where $a(x) \geq 0$ for a.e.~$x\in J$ and also
\begin{equation*}
\int_{0}^{T} a(x) ~\!dx < 0 < \int_{J} a(x) ~\!dx.
\end{equation*}
Suppose that $g(s)$ is a smooth function satisfying $(g_{1})$ and such that
\begin{equation*}
g'(0) = 0 < g'(+\infty) < +\infty.
\end{equation*}
Then there exists $\nu^{*} > 0$ such that for each $\nu > \nu^{*}$ the problem
\begin{equation*}
\begin{cases}
\, u'' + \nu \, a(x) g(u) = 0, \quad 0 < x < T, \\
\, {\mathscr{B}}(u,u') = \underline{0}
\end{cases}
\leqno{\hspace*{2.2pt}({\mathscr{P}}_{\nu})}
\end{equation*}
has at least one positive solution.
\end{theorem}

This theorem is sharp in the sense that we can also show that there are no positive solutions
to problem $({\mathscr{P}}_{\nu})$ if the parameter $\nu > 0$ is small (see Corollary~\ref{cor-3.7} and Proposition~\ref{prop-3.1}).

The above results are related to some recent theorems (cf.~\cite{Bo-2012, GrKoWa-2008, MaReTo-2014, To-2003})
concerning the existence of positive or sign-changing periodic solutions for
second order equations of the form
\begin{equation}\label{eq-1.5}
u'' + F(x,u) = 0,
\end{equation}
with $F(x,0) \equiv 0$. Although equation \eqref{eq-1.5} has a more general form than the equation in $({\mathscr{P}})$, our hypotheses
on $g(s)$ imply that $(\partial F/\partial u)(x,0) \equiv 0$,
a condition that does not allow to apply some of these results
or, when they can be applied, it involves solutions of non-constant sign.

\medskip

The original proof of Theorem~\ref{th-1.1} was obtained in \cite{GaHaZa-2003}
by a technique based on the construction of a pair of non well-ordered
lower and upper solutions. In our recent paper \cite{FeZa-pre2014} we have provided an extension of such a result (also with respect
to the multiplicity of solutions), using a topological degree approach. In \cite{FeZa-pre2014}, thanks to the fact that the operator
$u\mapsto -u''$  (subject to the Dirichlet boundary conditions) is invertible, we write
\eqref{eq-1.3} as an equivalent fixed point problem in a suitable Banach space and apply directly some degree theoretical arguments.
With respect to problem $({\mathscr{P}})$, the linear differential operator $u\mapsto -u''$ has a nontrivial kernel made by the constant functions.
In such a situation the operator is not invertible and we cannot proceed in the same manner as described above.
A possibility, already exploited in \cite{BeCaDoNi-1994}, is that of perturbing
the linear differential operator to a new one which can be inverted and then
recover the original equation via a limiting process and some careful estimates on the solutions.
In our case, we have found it very useful to apply the coincidence degree theory developed by J.~Mawhin,
which allows to study equations of the form $Lu = Nu$, where $L$ is a linear operator with nontrivial kernel and
$N$ is a nonlinear one. The use of the coincidence degree in the search for positive (periodic)
solutions is a widely used technique. For instance, in \cite{GaSa-1982, Sc-1976} a
coincidence theory on positive cones was initiated and developed, with applications to the search of
nontrivial non-negative periodic solutions. Our approach, however, is different and uses the classical technique of
extending the nonlinearity on the negative reals and, subsequently, proving that the nontrivial solutions are
positive, via a maximum principle (see Lemma~\ref{Maximum principle}). The existence of nontrivial solutions for the modified equation
follows by the additivity property of the coincidence degree, showing that the coincidence degree is equal to $1$ on small balls
(this follows from the condition on $\int a(x)$) and it is $0$ on large balls (this follows from $g_{\infty} >  \max \lambda_{1}^{i}$).
The advantage in using a topological degree approach lies also on the fact that, once we have found an open bounded set where the
degree is non-zero, we know that such a result is stable under small perturbations of the operator. Thus our theorems also
apply to equations which are small perturbations of the equation in problem $({\mathscr P}).$ For example, we
could even add to the equation small terms of a functional form, such as terms of (non-local) integral type or with a delay.
Of course, in such a case, to provide positive solutions, one should look for a suitable maximum principle.
In particular, Theorem~\ref{th-1.2} holds for equations as
\begin{equation*}
u'' + (a(x) + \varepsilon) g(u) = \delta u
\end{equation*}
for $|\varepsilon|$ and $|\delta|$ small enough.
\medskip

The plan of the paper is the following. In Section~\ref{section2} we recall some basic facts about Mawhin's coincidence degree
and we state two lemmas for the computation of the degree (see Lemma~\ref{lemma_Mawhin} and Lemma~\ref{lemma-2.3}).
These results are then applied in the second part of the same section to provide an existence theorem (see Theorem~\ref{th-2.1})
for positive solutions for a general problem of the form
\begin{equation*}
\begin{cases}
\, u'' + f(x,u,u') = 0, \quad 0 < x < T, \\
\, {\mathscr{B}}(u,u') = \underline{0}.
\end{cases}
\end{equation*}
The results of Section~\ref{section2} are then employed in Section~\ref{section3} in order to obtain two main results
for problem $({\mathscr{P}})$ under different conditions on the behavior of $g(s)$ near zero
(see Theorem~\ref{MainTheorem} and Theorem~\ref{MainTheorem2}). Various corollaries and applications are also derived.
In Section~\ref{section4} we present two different applications where we treat separately the Neumann and the periodic
problem. More precisely, in Section~\ref{section4.1} we prove an existence result of positive
radially symmetric solutions for a superlinear PDE subject to Neumann boundary conditions, while in
Section~\ref{section4.2} we provide positive periodic solutions to a Li\'{e}nard type equation.
We stress that in this latter case we can give an application of our method to a non-variational
setting, indeed the associated equation has not an Hamiltonian structure.
Throughout the paper we focus our study only to the case of the existence of nontrivial solutions. It seems reasonable to combine the
methods recently developed in \cite{FeZa-pre2014} with those of the present article in order to achieve multiplicity results of positive solutions.
This is briefly discussed in Section~\ref{section5}.
Some basic facts and technical estimates required for the proof of the main results are borrowed from our paper
\cite{FeZa-pre2014}. We insert them (with the details of the proofs) in an appendix (Section~\ref{Appendix}) for the reader's convenience.

\medskip

We end this introductory section with some definitions used in the article.
We denote by
\begin{equation*}
a^{+}(x) = \max\{a(x),0\} \quad \text{ and } \quad a^{-}(x) = \max\{-a(x),0\}
\end{equation*}
the \textit{positive part} and the \textit{negative part} of $a(x)$, respectively.

A function $g\colon\mathbb{R}^{+}\to\mathbb{R}^{+}$ satisfying $(g_{1})$ is \textit{regularly oscillating at zero} if
\begin{equation*}
\lim_{\substack{s\to0^{+} \\ \omega\to1}}\dfrac{g(\omega s)}{g(s)}=1.
\end{equation*}
This definition is the natural transposition for $s\to 0^{+}$ of the usual definition of
\textit{regularly oscillating (at infinity)} considered by several authors (see \cite{Be-1983}).
Regular oscillating functions are a class of maps
related to the study of Karamata regular variation theory and its many ramifications (cf.~\cite{BiGoTe-1987, Se-1976}).
They naturally appear in many different areas of real analysis like probability theory and qualitative theory
of ODEs (see \cite[\S~1]{DjTo-2001} for a brief historical survey about this subject).

\section{Abstract setting}\label{section2}

In the first part of this section we recall and summarize some basic properties of
Mawhin's coincidence degree which are useful for our applications.
We refer to \cite{GaMa-1977, Ma-1972, Ma-1979, Ma-1993} for more details, references and applications.
Then, in the second part of the section, we provide an existence result for the second order boundary value problem
\begin{equation}\label{BVP-sec2}
\begin{cases}
\, u'' + f(x,u,u') = 0, \quad 0 < x < T, \\
\, {\mathscr{B}}(u,u') = \underline{0},
\end{cases}
\end{equation}
which includes $({\mathscr{P}})$ as well as the case of more general nonlinear terms.

\subsection{Basic facts about Mawhin's coincidence degree}\label{section2.1}

Let $X$ and $Z$ be real Banach spaces and let $L \colon X \supseteq \text{\rm dom}\,L \to Z$ be a linear Fredholm mapping of index zero.
We denote by $\ker L = L^{-1}(0)$ the \textit{kernel} or null-space of $L$ and by $\text{\rm Im}\,L\subseteq Z$ the range or
\textit{image} of $L$. We fix a pair $(P,Q)$ of linear continuous projections with $P \colon X \to \ker L$ and
$Q \colon Z \to \text{\rm coker}\,L \subseteq Z$,
where $\text{\rm coker}\,L \cong Z/\text{\rm Im}\,L$ is the complementary subspace of $\text{\rm Im}\,L$ in $Z$.
The linear subspace $\ker P \subseteq X$ is the complementary subspace of $\ker L$ in $X$. Accordingly, we have
the decomposition in direct sums:
\begin{equation*}
X = \ker L \oplus \ker P, \qquad Z = \text{\rm Im}\,L \oplus \text{\rm Im}\,Q.
\end{equation*}
We denote by
\begin{equation*}
K_{P} \colon \text{\rm Im}\,L \to \text{\rm dom}\,L \cap \ker P
\end{equation*}
the right inverse of $L$, i.e.~$L K_{P}(w) = w$ for each $w\in \text{\rm Im}\,L$.
By the assumption that $L$ is a Fredholm mapping of index zero, we have that $\text{\rm Im}\,L$ is a closed subspace of $Z$ and
$\ker L$ and $\text{\rm coker}\,L$ are finite dimensional vector spaces of the same dimension.
We also fix an orientation on these spaces and take a linear
(orientation-preserving) isomorphism $J \colon \text{\rm coker}\,L \to \ker L$.

Let $N \colon X \to Z$ be possibly nonlinear operator and consider the \textit{coincidence equation}
\begin{equation}\label{eq-2.2}
Lu = Nu,\quad u\in \text{\rm dom}\,L.
\end{equation}
According to \cite{Ma-1979}, equation \eqref{eq-2.2} is equivalent to the fixed point problem
\begin{equation}\label{eq-2.3}
u = \Phi(u):= Pu + JQNu + K_{P}(Id-Q)Nu, \quad u\in X.
\end{equation}
Mawhin's coincidence degree theory allows to apply Leray-Schauder degree to the operator equation \eqref{eq-2.3},
thus providing a way to solve equation \eqref{eq-2.2} when $L$ is not invertible. To this aim,
we add some structural assumptions on $N$, in order to have a completely continuous right-hand side in \eqref{eq-2.3}.
More precisely, we suppose that the operator $N$ is \textit{$L$-completely continuous}, namely $N$ is continuous and,
for each bounded set $B\subseteq X$, it follows that $QN(B)$ and $K_{P}(Id-Q)N(B)$ are relatively compact sets. A typical situation in
which the $L$-complete continuity of $N$ is satisfied occurs when $N$ is continuous, maps bounded sets to bounded sets
and $K_{P}$ is a compact linear operator.

Let $\Omega\subseteq X$ be an open and bounded set such that
\begin{equation*}
Lu \neq Nu, \quad \forall \, u\in \text{\rm dom}\,L \cap \partial\Omega.
\end{equation*}
In this case, the \textit{coincidence degree of $L$ and $N$ in $\Omega$} is defined as
\begin{equation*}
D_{L}(L-N,\Omega):= \text{deg}(Id-\Phi,\Omega,0),
\end{equation*}
where ``\text{deg}'' denotes the Leray-Schauder degree. In the sequel we also denote
by ``$d_{B}$'' the (finite dimensional) Brouwer degree.
A remarkable result from coincidence degree theory guarantees that $D_{L}$ is independent on the choice of the projectors $P$ and $Q$.
Moreover, it is also independent of the choice of the linear isomorphism $J$, provided that we have fixed
an orientation on $\ker L$ and $\text{\rm coker}\,L$ and considered for $J$ only orientation-preserving isomorphisms.
This generalized degree has all the usual properties of Brouwer and Leray-Schauder degree, like additivity/excision and
homotopic invariance. In particular, equation \eqref{eq-2.2} has at least one solution in $\Omega$ if $D_{L}(L-N,\Omega)\neq0$.

The following result is of crucial importance in order to compute the coincidence degree.
It relates the coincidence degree to the finite dimensional Brouwer degree of the operator $N$ projected
into $\ker L$. It was given in \cite{Ma-1972} in its abstract form and, previously, in \cite{Ma-1969} in the context of periodic problems
for ODEs.

\begin{lemma}[Mawhin, 1969-1972]\label{lemma_Mawhin}
Let $L$ and $N$ be as above and let $\Omega\subseteq X$ be an open and bounded set.
Suppose that
\begin{equation*}
Lu \neq \vartheta Nu, \quad \forall \, u\in \text{\rm dom}\,L \cap \partial\Omega, \; \forall \, \vartheta\in\mathopen{]}0,1\mathclose{]},
\end{equation*}
and
\begin{equation*}
QN(u)\neq0, \quad \forall\, u\in \partial\Omega \cap \ker L.
\end{equation*}
Then
\begin{equation*}
D_{L}(L-N,\Omega) = d_{B}(-JQN|_{\ker L},\Omega \cap \ker L,0).
\end{equation*}
\end{lemma}

\begin{proof}
We give only a sketch of the proof for the reader's convenience. For the missing details, see~\cite{Ma-1979}.
Consider the operator $\Phi_{\vartheta}$ defined as
\begin{equation*}
\Phi_{\vartheta}(u):= Pu + JQNu + \vartheta K_{P}(Id-Q)Nu, \quad \text{for } \; \vartheta\in\mathopen{[}0,1\mathclose{]},
\end{equation*}
and observe that $\Phi_{1} = \Phi$ and $\Phi_{0}$ has finite dimensional range in $\ker L$.
The assumptions of the lemma imply that $u\neq\Phi_{\vartheta}u$, for all $u\in \partial\Omega$
and $\vartheta\in \mathopen{[}0,1\mathclose{]}$.
The homotopic invariance and the reduction property of the Leray-Schauder degree then give
\begin{equation*}
\begin{aligned}
D_{L}(L-N,\Omega) &= \text{deg}(Id-\Phi_{1},\Omega,0) = \text{deg}(Id-\Phi_{0},\Omega,0)\\
&= d_{B}(-JQN|_{\ker L},\Omega \cap \ker L,0).
\end{aligned}
\end{equation*}
Hence the result is proved. See~\cite{Ma-2008} for an interesting discussion on the reduction formula in the context of coincidence degree.
\end{proof}

\medskip

A typical degree theoretic approach in order to prove the existence of nontrivial solutions consists into showing that
the degree changes from small balls $B(0,r)$ to large balls $B(0,R)$, so that the additivity/excision property of the
degree ensures the existence of a solution in $B(0,R)\setminus B[0,r]$. From this point of view,
results ensuring that the degree is zero on some domains may be useful for the applications. In this context we
present the next result which is a simple adaptation to our setting of a well know lemma (cf.~\cite{Nu-1973}).

\begin{lemma}\label{lemma-2.2}
Let $L$ and $N$ be as above and let $\Omega\subseteq X$ be an open and bounded set.
Suppose that $v\neq0$ is a vector such that
\begin{equation*}
Lu \neq Nu + \alpha v, \quad \forall \, u\in \text{\rm dom}\,L \cap \partial\Omega, \; \forall\, \alpha \geq 0.
\end{equation*}
Then
\begin{equation*}
D_{L}(L-N,\Omega) = 0.
\end{equation*}
\end{lemma}

\begin{proof}
First of all, we observe that $u\in \text{\rm dom}\, L$ is a solution of the equation
$L u = N u + \alpha v$ if and only if $u\in X$ is a solution of
\begin{equation}\label{eq-2.6}
u= \Phi u + \alpha v^{*}, \quad \text{with }\; v^{*}:= JQ v + K_{P}(Id-Q)v.
\end{equation}
We claim that $v^{*} \neq 0$. Indeed, if $v^{*} = 0$, then $Q v = 0$ and also $K_{P} v = 0$. Hence, $v\in \text{\rm Im}\,L$
and therefore $v = LK_{P} v = 0$, a contradiction. Thus the claim is proved.

Since $\Phi$ is compact on the bounded set $\overline{\Omega}$, we have that
\begin{equation*}
M:= \sup_{u\in \overline{\Omega}} \|u-\Phi u\| < \infty.
\end{equation*}
We conclude that, if we fix any number
\begin{equation*}
\alpha_{0} > \dfrac{M}{\|v^{*}\|},
\end{equation*}
then \eqref{eq-2.6} has no solutions on $\overline{\Omega}$ for all $\alpha =\alpha_{0}$
(furthermore, there are no solutions also for $\alpha \geq \alpha_{0}$).

By the homotopic invariance of the coincidence degree (using $\alpha\in \mathopen{[}0,\alpha_{0}\mathclose{]}$ as a parameter), we find
\begin{equation*}
D_{L}(L-N,\Omega) = \text{deg}(Id-\Phi,\Omega,0) = \text{deg}(Id-\Phi - \alpha_{0} v^{*},\Omega,0) = 0.
\end{equation*}
Hence the result is proved.
\end{proof}

{} From the proof of Lemma~\ref{lemma-2.2} it is clear that the following variant holds.

\begin{lemma}\label{lemma-2.3}
Let $L$ and $N$ be as above and let $\Omega\subseteq X$ be an open and bounded set.
Suppose that there exist a vector $v\neq0$ and a constant $\alpha_{0} > 0$ such that
\begin{equation*}
Lu \neq Nu + \alpha v,
\quad \forall \, u\in \text{\rm dom}\,L \cap \partial\Omega, \; \forall\, \alpha\in \mathopen{[}0,\alpha_{0}\mathclose{]},
\end{equation*}
and
\begin{equation*}
Lu \neq Nu + \alpha_{0} v, \quad \forall \, u\in \text{\rm dom}\,L \cap \Omega.
\end{equation*}
Then
\begin{equation*}
D_{L}(L-N,\Omega) = 0.
\end{equation*}
\end{lemma}

\subsection{An existence result for problem \eqref{BVP-sec2}}\label{section2.2}

Throughout this section, by ${\mathscr{B}}(u,u') = \underline{0}$ we mean the Neumann or the periodic boundary conditions
on a fixed interval $\mathopen{[}0,T\mathclose{]}$.

Let $X:= \mathcal{C}^{1}(\mathopen{[}0,T\mathclose{]})$ be the Banach space of continuously differentiable real valued functions $u(x)$
defined on $\mathopen{[}0,T\mathclose{]}$ endowed with the norm
\begin{equation*}
\|u\|:= \|u\|_{\infty} + \|u'\|_{\infty}
\end{equation*}
and let $Z:= L^{1}(\mathopen{[}0,T\mathclose{]})$ be the space of Lebesgue integrable functions
defined on $\mathopen{[}0,T\mathclose{]}$ with the $L^{1}$-norm (denoted by $\|\cdot\|_{L^{1}}$).

We define $L\colon \text{\rm dom}\,L \to Z$ as
\begin{equation*}
(Lu)(x) := -u''(x),\quad x\in \mathopen{[}0,T\mathclose{]},
\end{equation*}
and take as $\text{\rm dom}\,L \subseteq X$ the vector subspace
\begin{equation*}
\text{\rm dom}\,L := \Bigl{\{} u\in X \colon u' \in \text{AC} \text{ and } {\mathscr{B}}(u,u') = \underline{0} \Bigr{\}},
\end{equation*}
where $u' \in \text{AC}$ means that $u'$ is absolutely continuous.
In this case, $\ker L \equiv {\mathbb{R}}$ is made by the constant functions and
\begin{equation*}
\text{\rm Im}\,L = \biggl\{w\in Z \colon \int_{0}^{T} w(x)~\!dx=0 \biggr\}.
\end{equation*}
A natural choice of the projections is given by
\begin{equation*}
P, Q \colon u \mapsto \dfrac{1}{T}\int_{0}^{T}u(x)~\!dx,
\end{equation*}
so that $\text{\rm coker}\,L \equiv {\mathbb{R}}$ and $\ker P$
is given by the continuously differentiable functions with mean value zero.
With such a choice of the projection, the right inverse linear operator $K_{P}$ is the map which,
to any $w\in L^{1}(\mathopen{[}0,T\mathclose{]})$ with $\int_{0}^{T} w = 0$, associates the unique solution $u(x)$ of
\begin{equation*}
u'' + w(x)= 0, \quad {\mathscr{B}}(u,u') = \underline{0}, \quad \int_{0}^{T} u(x) ~\!dx = 0.
\end{equation*}
Finally, we take as a linear isomorphism $J \colon \text{\rm coker}\,L \to \ker L$ the identity in ${\mathbb{R}}$.

We are ready now to introduce the nonlinear operator $N \colon X \to Z$. First we give some
assumptions on $f(x,s,\xi)$ which will be considered throughout the section.

Let $f\colon\mathopen{[}0,T\mathclose{]}\times {\mathbb{R}}^{+} \times {\mathbb{R}} \to {\mathbb{R}}$
be a $L^{p}$-Carath\'{e}odory function, for some $1\leq p \leq \infty$ (cf.~\cite{Ha-1980}), satisfying the following conditions
\begin{itemize}
\item[$(f_{1})$] \textit{$f(x,0,\xi) = 0$, for a.e.~$x\in \mathopen{[}0,T\mathclose{]}$ and for all $\xi\in {\mathbb{R}}$;}
\item[$(f_{2})$] \textit{there exists a non-negative function $k\in L^{1}(\mathopen{[}0,T\mathclose{]})$ and a constant $\rho > 0$ such that
\begin{equation*}
|f(x,s,\xi)| \leq k(x)(|s| + |\xi|),
\end{equation*}
for a.e.~$x\in \mathopen{[}0,T\mathclose{]}$, for all $0\leq s \leq \rho$ and $|\xi| \leq \rho$.}
\end{itemize}
Besides the above hypotheses, we suppose also that $f(x,s,\xi)$ satisfies a Bernstein-Nagumo type condition
in order to have a priori bounds on $|u'(x)|$ whenever bounds on $u(x)$ are obtained.
Typically, Bernstein-Nagumo assumptions are expressed in terms of growth restrictions on $f(x,s,\xi)$ with respect to the $\xi$-variable.
However, depending on the given boundary value problems and on the nonlinearity, more general conditions can be considered, too.
The interested reader can find in \cite{Ma-1974} a very general discussion for the periodic problem
(cf.~\cite{Za-1987} for a broad list of references). See also \cite{LoSc-2012} and \cite{Ma-1981} for interesting remarks
and applications to different boundary value problems.
For the purposes of the present paper, we do not consider the more general situation
and we confine ourselves to the classical estimate for the $L^{p}$-Carath\'{e}odory setting given in \cite[\S~4.4]{DCHa-2006}.
Accordingly, we assume that
\begin{itemize}
\item[$(f_{3})$] \textit{
for each $\eta > 0$ there exists a continuous function
\begin{equation*}
\phi = \phi_{\eta} \colon {\mathbb{R}}^{+} \to {\mathbb{R}}^{+}, \quad \text{with } \;
\displaystyle{\int^{\infty} \dfrac{\xi^{\frac{p-1}{p}}}{\phi(\xi)}~\!d\xi = \infty},
\end{equation*}
and a function $\psi = \psi_{\eta}\in L^{p}(\mathopen{[}0,T\mathclose{]},{\mathbb{R}}^{+})$ such that
\begin{equation*}
|f(x,s,\xi)| \leq \psi(x)\phi(|\xi|), \quad \text{for a.e. } x\in \mathopen{[}0,T\mathclose{]},
\; \forall \, s \in \mathopen{[}0,\eta\mathclose{]},
\; \forall \, \xi\in {\mathbb{R}}.
\end{equation*}}
\end{itemize}
For technical reasons, when dealing with Nagumo functions $\phi(\xi)$ as above, we always assume further that
\begin{equation*}
\liminf_{\xi\to +\infty} \phi(\xi) > 0.
\end{equation*}
This prevents the possibility of pathological examples like that in \cite[p.~46--47]{DCHa-2006} and does not affect our applications.

\medskip

As a first step we extend $f$ to a Carath\'{e}odory function $\tilde{f}$ defined on $\mathopen{[}0,T\mathclose{]}\times {\mathbb{R}}^{2}$,  by setting
\begin{equation*}
\tilde{f}(x,s,\xi):=
\begin{cases}
\, f(x,s,\xi), & \text{if } s \geq 0;\\
\, - s,      & \text{if } s \leq 0;
\end{cases}
\end{equation*}
and denote by
$N \colon X \to Z$ the Nemytskii operator induced by $\tilde{f}$, that is
\begin{equation*}
(Nu)(x):= \tilde{f}(x,u(x),u'(x)), \quad x\in \mathopen{[}0,T\mathclose{]}.
\end{equation*}
In this setting, $u$ is a solution of the coincidence equation
\begin{equation}\label{eq-2.7}
Lu = Nu, \quad u\in \text{\rm dom}\,L,
\end{equation}
if and only if it is a solution to the boundary value problem
\begin{equation}\label{eq-2.8}
\begin{cases}
\, u'' + \tilde{f}(x,u,u') = 0, \quad 0 < x < T, \\
\, {\mathscr{B}}(u,u') = \underline{0}.
\end{cases}
\end{equation}
Moreover, from the definition of $\tilde{f}$ for $s \leq 0$ and conditions $(f_{1})$ and $(f_{2})$, one can easily check
by a maximum principle argument (see Lemma~\ref{Maximum principle} and Remark~\ref{rem-6.1})
that if $u \not\equiv 0$, then $u(x)$ is strictly positive and hence a
(positive) solution of problem~\eqref{BVP-sec2}.

Now, as an application of Lemma~\ref{lemma_Mawhin} and Lemma~\ref{lemma-2.3}, we have the following result.

\begin{theorem}\label{th-2.1}
Assume $(f_{1})$, $(f_{2})$, $(f_{3})$ and suppose that there exist two constants $r, R > 0$, with $r\neq R$,
such that the following hypotheses hold.
\begin{itemize}
\item [$(H_{r})$]
The average condition
\begin{equation*}
\int_{0}^{T} f(x,r,0)~\!dx < 0
\end{equation*}
is satisfied.
Moreover, any solution $u(x)$ of the boundary value problem
\begin{equation}\label{eq-2.9}
\begin{cases}
\, u'' + \vartheta {f}(x,u,u') = 0 \\
\, {\mathscr{B}}(u,u') = \underline{0},
\end{cases}
\end{equation}
for $0 < \vartheta \leq 1$, such that $u(x) > 0$ on $\mathopen{[}0,T\mathclose{]}$,
satisfies $\|u\|_{\infty} \neq r$.
\item [$(H_{R})$]
There exist a non-negative function $v\in L^{p}(\mathopen{[}0,T\mathclose{]})$
with $v\not\equiv 0$ and a constant $\alpha_{0} > 0$, such that
every solution $u(x)\geq0$ of the boundary value problem
\begin{equation}\label{eq-2.10}
\begin{cases}
\, u'' + {f}(x,u,u') + \alpha v(x)= 0 \\
\, {\mathscr{B}}(u,u') = \underline{0},
\end{cases}
\end{equation}
for $\alpha \in \mathopen{[}0,\alpha_{0}\mathclose{]}$,
satisfies $\|u\|_{\infty} \neq R$. Moreover,
there are no solutions $u(x)$ of \eqref{eq-2.10} for $\alpha = \alpha_{0}$ with $0 \leq u(x) \leq R$,
for all $x\in \mathopen{[}0,T\mathclose{]}$.
\end{itemize}
Then problem \eqref{BVP-sec2} has at least a positive solution $u(x)$ with
\begin{equation*}
\min\{r,R\} < \max_{x\in \mathopen{[}0,T\mathclose{]}} u(x) < \max\{r,R\}.
\end{equation*}
\end{theorem}

\begin{proof}
As we have already observed, from the choice of the spaces $X$, $\text{\rm dom}\, L$, $Z$ and the operators $L \colon u\mapsto -u''$ and $N$
(the Nemytskii operator induced by $\tilde{f}$), we have that \eqref{eq-2.7} is equivalent to the boundary value problem
\eqref{eq-2.8}. All the structural assumptions required by Mawhin's theory (that is $L$ is Fredholm of index zero and $N$ is
$L$-completely continuous) are satisfied by standard facts (see \cite{Ma-1979}).

For the proof, we confine ourselves to the case
\begin{equation*}
0 < r < R,
\end{equation*}
which is the interesting one for our applications.
The case in which $0 < R < r$ can be studied with minor changes in the proof
and it will be briefly described at the end.

The coincidence equation
\begin{equation}\label{eq-2.11}
Lu = \vartheta Nu, \quad u\in \text{\rm dom}\,L,
\end{equation}
is equivalent to
\begin{equation}\label{eq-2.12}
\begin{cases}
\, u'' + \vartheta \tilde{f}(x,u,u') = 0 \\
\, {\mathscr{B}}(u,u') = \underline{0}.
\end{cases}
\end{equation}
Let $u$ be any solution of \eqref{eq-2.11} for some $\vartheta > 0$. From the definition of $\tilde{f}$ for $s\leq 0$
and the maximum principle, we have that $u(x)\geq 0$ for every $x\in \mathopen{[}0,T\mathclose{]}$
and hence $u$ is a solution of \eqref{eq-2.9}. Moreover, by $(f_{2})$, if $u\not\equiv 0$, then $u(x) > 0$
for all $x\in \mathopen{[}0,T\mathclose{]}$. See also the Appendix.

According to condition $(f_{3})$, let $\phi = \phi_{r} \colon {\mathbb{R}}^{+} \to {\mathbb{R}}^{+}$
and $\psi = \psi_{r}\in L^{p}(\mathopen{[}0,T\mathclose{]})$
be such that $|f(x,s,\xi)|\leq \psi(x)\phi(|\xi|)$,
for a.e.~$x\in \mathopen{[}0,T\mathclose{]}$, for all $s\in \mathopen{[}0,r\mathclose{]}$ and $\xi\in {\mathbb{R}}$.
By Nagumo lemma (cf.~\cite[\S~4.4, Proposition~4.7]{DCHa-2006}), there exists a constant $M= M_{r} > 0$
(depending on $r$, as well as on $\phi$ and $\psi$,
but not depending on $u(x)$ and $\vartheta\in\mathopen{]}0,1\mathclose{]}$) such that any solution of \eqref{eq-2.12}
or, equivalently, any (non-negative) solution of \eqref{eq-2.9} (for some $\vartheta\in\mathopen{]}0,1\mathclose{]}$)
satisfying $\|u\|_{\infty} \leq r$ is such that $\|u'\|_{\infty} < M_{r}$.
Hence, condition $(H_{r})$ implies that, for the open and bounded set $\Omega_{r}$ in $X$ defined as
\begin{equation*}
\Omega_{r}:= \bigl{\{}u\in X \colon \|u\|_{\infty} < r, \; \|u'\|_{\infty} < M_{r}\bigr{\}},
\end{equation*}
it holds that
\begin{equation*}
Lu \neq \vartheta Nu,
\quad \forall\, u\in \text{\rm dom}\,L \cap \partial\Omega_{r}, \; \forall\, \vartheta\in \mathopen{]}0,1\mathclose{]}.
\end{equation*}
Consider now $u\in \partial\Omega_{r} \cap \ker L$. In this case, $u\equiv k\in {\mathbb{R}}$, with $|k| = r$, and
\begin{equation*}
-JQN u = - \dfrac{1}{T} \int_{0}^{T} \tilde{f}(x,k,0)~\!dx.
\end{equation*}
Notice also that $\Omega_{r}\cap \ker L = \mathopen{]}-r,r\mathclose{[}$.

By the definition of $\tilde{f}$ for $s\leq 0$, we have that
\begin{equation*}
f^{\#}(s):= - \dfrac{1}{T} \int_{0}^{T} \tilde{f}(x,s,0)~\!dx =
\begin{cases}
\, - \dfrac{1}{T} \displaystyle\int_{0}^{T} f(x,s,0)~\!dx, & \text{if } s >  0; \\
\, s,          & \text{if } s \leq 0.
\end{cases}
\end{equation*}
Therefore, $QNu\neq 0$ for each $u\in \partial\Omega_{r} \cap \ker L$ and, moreover,
\begin{equation*}
d_B(f^{\#}, \mathopen{]}-r,r\mathclose{[},0) =1,
\end{equation*}
since $f^{\#}(-r) < 0 < f^{\#}(r)$.
By Lemma~\ref{lemma_Mawhin} we conclude that
\begin{equation}\label{eq-deg1}
D_{L}(L-N,\Omega_{r}) = 1.
\end{equation}

Now we study the operator equation
\begin{equation}\label{eq-2.14}
Lu = Nu + \alpha v, \quad u\in \text{\rm dom}\,L,
\end{equation}
for some $\alpha \geq 0$, with  $v$ as in $(H_{R})$. This equation
is equivalent to
\begin{equation}\label{eq-2.15}
\begin{cases}
\, u'' +  \tilde{f}(x,u,u') + \alpha v(x) = 0\\
\, {\mathscr{B}}(u,u') = \underline{0}.
\end{cases}
\end{equation}
Let $u$ be any solution of \eqref{eq-2.14} for some $\alpha \geq 0$. From the definition of $\tilde{f}$ for $s\leq 0$
and the maximum principle, we have that $u(x)\geq 0$ for every $x\in \mathopen{[}0,T\mathclose{]}$ and hence $u$ is a solution of
\eqref{eq-2.10}.

According to condition $(f_{3})$, let $\phi = \phi_{R} \colon {\mathbb{R}}^{+} \to {\mathbb{R}}^{+}$
and $\psi = \psi_{R}\in L^{p}(\mathopen{[}0,T\mathclose{]})$ be such that $|f(x,s,\xi)|\leq \psi(x)\phi(|\xi|)$,
for a.e.~$x\in \mathopen{[}0,T\mathclose{]}$, for all $s\in \mathopen{[}0,R\mathclose{]}$ and $\xi\in {\mathbb{R}}$.
If we take $\alpha\in \mathopen{[}0,\alpha_{0}\mathclose{]}$, we obtain that
\begin{equation*}
|f(x,s,\xi) + \alpha v(x)| \leq \psi(x) \phi(|\xi|) + \alpha_{0} v(x) \leq \tilde{\psi}(x) \tilde{\phi}(|\xi|)
\end{equation*}
holds for a.e.~$x\in \mathopen{[}0,T\mathclose{]}$ and for all $s\in \mathopen{[}0,R\mathclose{]}$ and
$\xi\in {\mathbb{R}}$, with
\begin{equation*}
\tilde{\psi}(x):= \psi(x) + \alpha_{0} v(x)\quad \text{ and } \quad \tilde{\phi}(\xi):= \phi(\xi) + 1.
\end{equation*}
Observe also that $\tilde{\psi}\in L^{p}(\mathopen{[}0,T\mathclose{]})$ and $\int^{\infty} \xi^{(p-1)/p}/\tilde{\phi}(\xi)~\!d\xi = \infty$.

By Nagumo lemma, there exists a positive constant $M= M_{R} > M_{r}$ (depending on $R$, as well as on $\phi$ and $\tilde{\psi}$,
but not depending on $u(x)$ and $\alpha\in\mathopen{[}0,\alpha_{0}\mathclose{]}$) such that any solution of \eqref{eq-2.15}
or, equivalently, any (non-negative) solution of \eqref{eq-2.10} (for some $\alpha\in\mathopen{[}0,\alpha_{0}\mathclose{]}$)
satisfying $\|u\|_{\infty} \leq R$ is such that $\|u'\|_{\infty} < M_{R}$.
Hence, condition $(H_{R})$ implies that, for the open and bounded set $\Omega_{R}$ in $X$ defined as
\begin{equation*}
\Omega_{R}:= \bigl{\{} u\in X \colon \|u\|_{\infty} < R, \; \|u'\|_{\infty} < M_{R} \bigr{\}},
\end{equation*}
it holds that
\begin{equation*}
Lu \neq  Nu +\alpha v ,
\quad \forall \, u\in \text{\rm dom}\,L \cap \partial\Omega_{R}, \; \forall\, \alpha\in \mathopen{[}0,\alpha_{0}\mathclose{]}.
\end{equation*}
Moreover, the last hypothesis in $(H_{R})$ also implies that
\begin{equation*}
Lu \neq  Nu +\alpha_{0} v, \quad \forall \, u\in \text{\rm dom}\, L \cap \Omega_{R}.
\end{equation*}
According to Lemma~\ref{lemma-2.3} we have that
\begin{equation}\label{eq-deg0}
D_{L}(L-N,\Omega_{R}) = 0.
\end{equation}
In conclusion, from \eqref{eq-deg1}, \eqref{eq-deg0} and the additivity property of the coincidence degree, we find that
\begin{equation*}
D_{L}(L-N, \Omega_{R} \setminus \text{\rm cl}(\Omega_{r})) = -1.
\end{equation*}
This ensures the existence of a (nontrivial) solution $\tilde{u}$ to \eqref{eq-2.7}
with $\tilde{u}\in \Omega_{R} \setminus \text{\rm cl}(\Omega_{r})$.
Since $\tilde{u}$ is a nontrivial solution of \eqref{eq-2.8}, by the (strong) maximum principle
(following from the definition of $\tilde{f}$ for $s\leq 0$, $(f_{1})$ and $(f_{2})$),
we have that $\tilde{u}$ is a solution of \eqref{BVP-sec2} with $\tilde{u}(x)>0$ for all $x\in \mathopen{[}0,T\mathclose{]}$.

\medskip

If we are in the case
\begin{equation*}
0 < R < r,
\end{equation*}
we proceed in an analogous manner. With respect to the previous situation, the only relevant changes are the following.
First we fix a constant $M = M_{R} > 0$ and, for the set $\Omega_{R}$, we obtain \eqref{eq-deg0}. As a next step,
we repeat the first part of the above proof, we fix a constant $M_{r} > M_{R}$ and, for the set $\Omega_{r}$, we obtain \eqref{eq-deg1}.
Now we have
\begin{equation*}
D_{L}(L-N, \Omega_{r} \setminus \text{\rm cl}(\Omega_{R})) = 1.
\end{equation*}
This ensures the existence of a (nontrivial) solution $\tilde{u}$ to \eqref{eq-2.7}
with $\tilde{u}\in \Omega_{r} \setminus \text{\rm cl}(\Omega_{R})$ and then we conclude as above, showing that
$\tilde{u}(x)>0$ for all $x\in \mathopen{[}0,T\mathclose{]}$ (by the strong maximum principle).
\end{proof}

\begin{remark}\label{rem-2.1}
If we take as boundary conditions the periodic ones, namely with
${\mathscr{B}}(u,u') = \underline{0}$ written as
\begin{equation*}
u(0) = u(T), \qquad u'(0) = u'(T),
\end{equation*}
then Theorem~\ref{th-2.1} holds true if, in place of the differential operator $u \mapsto - u''$,
we take a linear differential operator of the form $u \mapsto - (u'' + c u')$,
with $c\in {\mathbb{R}}$ a fixed constant.
$\hfill\lhd$
\end{remark}

\begin{remark}\label{rem-2.2}
The condition $(f_{2})$ is required only to assure that a non-negative solution is strictly positive. If we do not assume $(f_{2})$,
with the same proof, we can provide a variant of Theorem~\ref{th-2.1} in which we obtain the existence of nontrivial non-negative solutions.
In this case, condition $(H_{r})$ should be modified requiring that any $u(x) \geq 0$ satisfies $\|u\|_{\infty} \neq r$.
$\hfill\lhd$
\end{remark}

\section{Positive solutions of superlinear problems}\label{section3}

In this section we give an application of Theorem~\ref{th-2.1} to the existence of positive solutions for problem $({\mathscr{P}})$.
Throughout the section, we suppose that $g \colon {\mathbb{R}}^{+} \to {\mathbb{R}}^{+}$ is a continuous function such that
\begin{equation*}
g(0)=0, \qquad g(s)>0 \quad \text{for } \; s>0.
\leqno{\hspace*{2.2pt}(g_{1})}
\end{equation*}
The weight coefficient $a\colon\mathopen{[}0,T\mathclose{]}\to{\mathbb{R}}$ is a $L^{1}$-function such that
\begin{itemize}
\item[$(a_{1})$]
\textit{there exist $m \geq 1$ intervals $I_{1},\ldots,I_{m}$, closed and pairwise disjoint, such that
\begin{equation*}
\begin{aligned}
& a(x)\geq 0, \;\; \text{for a.e. } x\in I_{i}, \text{ with }\; a(x)\not\equiv 0 \text{ on } I_{i} \quad (i=1,\ldots,m);\\
& a(x)\leq 0, \;\; \text{for a.e. } x\in \mathopen{[}0,T\mathclose{]}\setminus\bigcup_{i=1}^{m} I_{i};
\end{aligned}
\end{equation*}}
\end{itemize}
\vspace*{-12pt}
\begin{equation*}
\bar{a}:=\dfrac{1}{T}\int_{0}^{T}a(x)~\!dx < 0.
\leqno{\hspace*{2.2pt}(a_{2})}
\end{equation*}

Let $\lambda_{1}^{i}$, $i=1,\ldots,m$, be the first eigenvalue of the eigenvalue problem
\begin{equation}\label{eq-eig}
\varphi'' + \lambda a(x) \varphi = 0, \quad \varphi|_{\partial I_{i}} = 0.
\end{equation}
\noindent From the assumptions on $a(x)$ in $I_{i}$ it clearly follows that $\lambda_{1}^{i} > 0$ for each $i=1,\ldots, m$.
In the sequel, if necessary, it will be not restrictive to label the intervals $I_{i}$ following the natural order given by the
standard orientation of the real line.

\begin{theorem}\label{MainTheorem}
Let $g(s)$ and $a(x)$ be as above. Suppose also that $g(s)$ is regularly oscillating at zero and satisfies
\begin{equation*}
\lim_{s\to 0^{+}}\dfrac{g(s)}{s}=0 \quad \text{ and } \quad
g_{\infty}:= \liminf_{s\to +\infty} \dfrac{g(s)}{s} > \max_{i=1,\ldots,m} \lambda_{1}^{i}.
\leqno{\hspace*{2.2pt}(g_{2})}
\end{equation*}
Then problem $({\mathscr{P}})$ has at least one positive solution.
\end{theorem}

\begin{proof}
In order to enter in the setting of Theorem~\ref{th-2.1} we define
\begin{equation*}
f(x,s,\xi) = f(x,s) := a(x)g(s)
\end{equation*}
and observe that $f$ is a $L^{1}$-Carath\'{e}odory function.
The basic hypotheses required on $f(x,s,\xi)$ are all satisfied.
In fact, $(f_{1})$ follows from $g(0) = 0$ and $(f_{2})$ is an obvious consequence of the fact that
$g(s)/s$ is bounded on a right neighborhood of $s=0$ and
$a\in L^{1}(\mathopen{[}0,T\mathclose{]})$. By the continuity of $g(s)$ and the integrability of $a(x)$,
the Nagumo condition $(f_{3})$ is trivially satisfied since $f$ does not depend on $\xi$.
Indeed, we can take $p=1$, $\phi(\xi)\equiv 1$ and $\psi(x) = |a(x)| \, \max_{0\leq s \leq \eta} g(s)$.

\medskip

\noindent
\textit{Verification of $(H_{r})$.} First of all, we observe that $(g_{1})$ and $(a_{2})$ imply that
\begin{equation}\label{eq-fs}
\int_{0}^{T} f(x,s,0)~\!dx < 0, \quad \forall \, s > 0.
\end{equation}
We claim that there exists $r_{0}>0$ such that for all $0<r\leq r_{0}$ and
for all $\vartheta\in\mathopen{]}0,1\mathclose{]}$ there are no solutions $u(x)$ of \eqref{eq-2.9}
such that $u(x)>0$ on $\mathopen{[}0,T\mathclose{]}$ and $\|u\|_{\infty}=r$.

By contradiction, suppose the claim is not true. Then for all $n\in{\mathbb{N}}$
there exist $0<r_{n} <1/n$, $\vartheta_{n}\in\mathopen{]}0,1\mathclose{]}$ and $u_{n}(x)$ solution of
\begin{equation}\label{eq-vartheta-n}
u''+\vartheta_{n} a(x)g(u)=0,\quad {\mathscr{B}}(u,u') = \underline{0},
\end{equation}
such that $u_{n}(x)>0$ on $\mathopen{[}0,T\mathclose{]}$ and $\|u_{n}\|_{\infty} = r_{n}$.

Integrating on $\mathopen{[}0,T\mathclose{]}$ the differential equation in \eqref{eq-vartheta-n} and using the boundary conditions,
we obtain
\begin{equation*}
0=-\int_{0}^{T}u''_{n}(x)~\!dx = \vartheta_{n}\int_{0}^{T}a(x)g(u_{n}(x))~\!dx .
\end{equation*}
Then
\begin{equation}\label{eq-3.4}
\int_{0}^{T}a(x)g(u_{n}(x))~\!dx = 0
\end{equation}
follows.
We define
\begin{equation*}
v_{n}(x):=\dfrac{u_{n}(x)}{\|u_{n}\|_{\infty}}
\end{equation*}
and, dividing \eqref{eq-vartheta-n} by $r_{n} = \|u_{n}\|_{\infty}$,  we get
\begin{equation}\label{eq-vn}
v''_{n}(x)+ \vartheta_{n}a(x)\dfrac{g(u_{n}(x))}{u_{n}(x)}v_{n}(x)=0.
\end{equation}
By the first assumption in $(g_{2})$, for every $\varepsilon>0$ there exists $\delta_{\varepsilon}>0$ such that
\begin{equation*}
0<\dfrac{g(s)}{s}<\varepsilon, \quad \forall \, 0<s<\delta_{\varepsilon}.
\end{equation*}
For $n > 1/\delta_{\varepsilon}$ we have $0 < u_{n}(x) \leq r_{n} < \delta_{\varepsilon}$ for all $x\in \mathopen{[}0,T\mathclose{]}$, so that
\begin{equation*}
0<\dfrac{g(u_{n}(x))}{u_{n}(x)}<\varepsilon, \quad \forall \, x\in \mathopen{[}0,T\mathclose{]}.
\end{equation*}
This proves that
\begin{equation}\label{eq-3.6}
\lim_{n\to \infty}\dfrac{g(u_{n}(x))}{u_{n}(x)}= 0, \quad \text{uniformly on } \mathopen{[}0,T\mathclose{]}.
\end{equation}

We fix $x_{0}\in \mathopen{[}0,T\mathclose{]}$ such that $v_{n}'(x_{0})=0$. With the Neumann boundary condition, we choose $x_{0}=0$ (or $x_{0}=T$),
while, in the periodic case, the existence of such a $x_{0}$ (possibly depending on $n$) is ensured by Rolle's theorem.
Integrating \eqref{eq-vn}, we have
\begin{equation*}
v_{n}'(x)= - \vartheta_{n}\int_{x_{0}}^{x}a(\xi)\dfrac{g(u_{n}(\xi))}{u_{n}(\xi)}v_{n}(\xi)~\!d\xi,
\end{equation*}
hence
\begin{equation}\label{eq-3.7}
\|v_{n}'\|_{\infty} \leq \int_{0}^{T}|a(x)|\dfrac{g(u_{n}(x))}{u_{n}(x)}~\!dx.
\end{equation}
Since $a\in L^{1}(\mathopen{[}0,T\mathclose{]})$, by \eqref{eq-3.6} and the dominated convergence theorem, we find that
$v_{n}'(x)\to 0$ (as $n\to \infty$) uniformly on $\mathopen{[}0,T\mathclose{]}$.

Since $\|v_{n}\|_{\infty}=1$, there exists $x_{1}\in \mathopen{[}0,T\mathclose{]}$ (possibly depending on $n$)
such that $v_{n}(x_{1})=1$. From
\begin{equation*}
v_{n}(x) = v_{n}(x_{1}) - \int_{x_{1}}^{x}v_{n}'(\xi)~\!d\xi,
\end{equation*}
we conclude that
\begin{equation}\label{eq-3.8}
\lim_{n\to \infty} v_{n}(x) = 1, \quad \text{uniformly on } \mathopen{[}0,T\mathclose{]}.
\end{equation}

Now, we write \eqref{eq-3.4} as
\begin{equation*}
0=\int_{0}^{T}a(x)g(u_{n}(x))~\!dx
=\int_{0}^{T} \Bigl{(} a(x)g(r_{n})+a(x)[g(r_{n}v_{n}(x))-g(r_{n})] \Bigr{)}~\!dx.
\end{equation*}
Since $g(r_{n})>0$, then
\begin{equation*}
-\dfrac{1}{T}\int_{0}^{T}a(x)~\!dx
=\dfrac{1}{T}\int_{0}^{T}a(x)\dfrac{g(r_{n}v_{n}(x))-g(r_{n})}{g(r_{n})}~\!dx.
\end{equation*}
Consequently, by $(a_{2})$,
\begin{equation*}
0<-\bar{a}\leq\dfrac{1}{T}\|a\|_{L^{1}} \, \max_{x\in \mathopen{[}0,T\mathclose{]}}\biggr{|}\dfrac{g(r_{n}v_{n}(x))}{g(r_{n})}-1\biggl{|}
= \dfrac{1}{T}\|a\|_{L^{1}} \, \biggr{|}\dfrac{g(r_{n}\omega_{n})}{g(r_{n})}-1\biggl{|},
\end{equation*}
where $\omega_{n} := v_{n}(x_{n})$, for a suitable choice of $x_{n}\in \mathopen{[}0,T\mathclose{]}$,
and also $\omega_{n} \to 1$ (as $n\to \infty$) by \eqref{eq-3.8}.
Using the fact that $g$ is regularly oscillating at zero, we obtain a contradiction as $n\to\infty$.

The claim is thus proved and, recalling also \eqref{eq-fs}, we have that $(H_{r})$ holds
for any $r\in \mathopen{]}0,r_{0}\mathclose{]}$.

\medskip

\noindent
\textit{Verification of $(H_{R})$.}
First of all, we fix a nontrivial function $v\in L^{1}(\mathopen{[}0,T\mathclose{]})$, with $v(x)\not\equiv 0$ on $\mathopen{[}0,T\mathclose{]}$,
such that
\begin{equation*}
\begin{aligned}
& v(x)\geq 0, \;\; \text{for a.e. } x\in \bigcup_{i=1}^{m} I_{i};\\
& v(x) = 0, \;\; \text{for a.e. } x\in \mathopen{[}0,T\mathclose{]}\setminus\bigcup_{i=1}^{m} I_{i}.
\end{aligned}
\end{equation*}
For example, as $v(x)$ we can take the characteristic function of the set
\begin{equation*}
A:=\bigcup_{i=1}^{m} I_{i}.
\end{equation*}
Secondly, we observe that $a(x) \geq 0$ on each interval $I_{i}$, so that
\begin{equation*}
\liminf_{s\to +\infty} \dfrac{f(x,s)}{s} \geq a(x) g_{\infty}, \quad \text{uniformly a.e. } x\in I_{i}.
\end{equation*}
Moreover, the second condition in $(g_{2})$ implies that the first eigenvalue of the eigenvalue problem
\begin{equation*}
\varphi''+\lambda g_{\infty}a(x) \varphi = 0, \quad \varphi|_{\partial I_{i}}=0,
\end{equation*}
is strictly less than $1$, for all $i=1,\ldots,m$.
Thus, we can apply Lemma~\ref{lemma_RJ} on each interval $I_{i}$, with
\begin{equation*}
J:=I_{i}, \quad h(x,s):=a(x)g(s), \quad q_{\infty}(x):=a(x)g_{\infty}.
\end{equation*}
Hence, for each $i=1,\ldots,m$, we obtain the existence of a constant $R_{I_{i}}>0$ such that
for each Carath\'{e}odory function $k\colon \mathopen{[}0,T\mathclose{]} \times{\mathbb{R}}^{+}\to {\mathbb{R}}$ with
\begin{equation*}
k(x,s)\geq a(x)g(s), \quad \text{a.e. } x\in I_{i}, \; \forall \, s\geq0,
\end{equation*}
every solution $u(x)\geq0$ of the BVP
\begin{equation}\label{BVP-3.9}
\begin{cases}
\, u''+ k(x,u)=0 \\
\, {\mathscr{B}}(u,u') = \underline{0}
\end{cases}
\end{equation}
satisfies $\max_{x\in I_{i}}u(x)<R_{I_{i}}$.

Then, we fix a constant $R > r_{0}$ (with $r_{0}$ coming from the first part of the proof) such that
\begin{equation}\label{eq-R}
R \geq \max_{i=1,\ldots,m}R_{I_{i}}
\end{equation}
and another constant $\alpha_{0} > 0$ such that
\begin{equation}\label{eq-3.11}
\alpha_{0}>\dfrac{\|a\|_{L^{1}}\,\max_{0\leq s \leq R}g(s)}{\|v\|_{L^{1}}}.
\end{equation}

We take $\alpha\in\mathopen{[}0,\alpha_{0}\mathclose{]}$.
We observe that any solution $u(x)\geq 0$ of problem \eqref{eq-2.10} is a solution of \eqref{BVP-3.9} with
\begin{equation*}
k(x,s) = a(x)g(s) + \alpha v(x).
\end{equation*}
By definition, we have that $k(x,s) \geq a(x)g(s)$ for a.e.~$x\in A$ and for all $s\geq 0$, and also $k(x,s) = a(x)g(s) \leq 0$
for a.e.~$x\in \mathopen{[}0,T\mathclose{]}\setminus A$ and for all $s\geq 0$.
By the convexity of the solutions of \eqref{BVP-3.9} on the intervals of $\mathopen{[}0,T\mathclose{]}\setminus A$, we obtain
\begin{equation*}
\max_{x\in \mathopen{[}0,T\mathclose{]}}u(x)=\max_{x\in A}u(x)
\end{equation*}
and, as an application of Lemma~\ref{lemma_RJ} on each of the intervals $I_{i}$, we conclude that
\begin{equation*}
\|u\|_{\infty} < R.
\end{equation*}
This proves the first part of $(H_{R})$.

It remains to verify that for $\alpha = \alpha_{0}$ defined in \eqref{eq-3.11} there are no solutions $u(x)$ of \eqref{eq-2.10} with
$0\leq u(x) \leq R$ on $\mathopen{[}0,T\mathclose{]}$. Indeed, if $u$ is a solution of \eqref{eq-2.10}, or
equivalently of
\begin{equation*}
\begin{cases}
\, u''+ a(x)g(u) + \alpha v(x) = 0 \\
\, {\mathscr{B}}(u,u') = \underline{0},
\end{cases}
\end{equation*}
with $0 \leq u(x) \leq R$, then, integrating on $\mathopen{[}0,T\mathclose{]}$ the differential equation
 and using the boundary conditions, we obtain
\begin{equation*}
\alpha \|v\|_{L^{1}} = \alpha \int_{0}^{T}v(x)~\!dx \leq \int_{0}^{T}|a(x)|g(u(x))~\!dx \leq \|a\|_{L^{1}}\,\max_{0\leq s \leq R}g(s),
\end{equation*}
which leads to a contradiction with respect to the choice of $\alpha_{0}$. Thus $(H_{R})$ is verified.

\medskip

\noindent
Having verified $(H_{r})$ and $(H_{R})$, the thesis follows from Theorem~\ref{th-2.1}.
\end{proof}

\begin{remark}\label{rem-3.1}
\noindent From the verification of condition $(H_{R})$ performed in the above proof, it is clear that
the assumption $g_{\infty} > \max \lambda_{1}^{i}$ is employed in connection with Lemma~\ref{lemma_RJ} (in the Appendix)
in order to obtain the a priori bounds $R_{I_{i}}$ on the intervals $I_{i}$. In turn, this step in the proof is based on a
Sturm comparison argument involving the eigenfunctions of \eqref{eq-eig}. If among the intervals $I_{i}$ there is one of the form
$I_{1} = \mathopen{[}0,\sigma\mathclose{]}$ or one of the form $I_{m} = \mathopen{[}\tau,T\mathclose{]}$ (both cases are also possible),
then the choice of the eigenvalues can be
made in a more refined manner in order to improve the lower bound for $g_{\infty}$. More precisely, whenever such a situation occurs,
we can proceed as follows.

\noindent
$\;\bullet\;$ For the Neumann problem, if $I_{1} = \mathopen{[}0,\sigma\mathclose{]}$ we take as $\lambda_{1}^{1}$
the first positive eigenvalue of the eigenvalue problem
\begin{equation*}
\varphi'' + \lambda a(x) \varphi = 0, \quad \varphi'(0) = \varphi(\sigma) = 0.
\end{equation*}
Similarly, if $I_{m} = \mathopen{[}\tau,T\mathclose{]}$ we take as $\lambda_{1}^{m}$ the first positive eigenvalue of the eigenvalue problem
\begin{equation*}
\varphi'' + \lambda a(x) \varphi = 0, \quad \varphi(\tau) = \varphi'(T) = 0.
\end{equation*}

\noindent
$\;\bullet\;$ For the periodic problem, we can extend the coefficient by $T$-periodicity on the
whole real line and, after a shift on the $x$-variable, we consider an equivalent problem where the
weight function is negative in a neighborhood of the endpoints. We try to clarify this concept with an example.
Suppose that we are interested in the search of $2\pi$-periodic solutions of equation
\begin{equation*}
u'' + (-k +\cos(x)) g(u) = 0,
\end{equation*}
where $ 0 < k < 1$ is a fixed constant.
In this case, setting the problem on the interval $\mathopen{[}0,2\pi\mathclose{]}$,
we should consider the eigenvalue problem \eqref{eq-eig} on the two intervals
$I_{1} = \mathopen{[}0,\arccos k\mathclose{]}$ and $I_{2} = \mathopen{[}2\pi-\arccos k,2\pi\mathclose{]}$, where the weight function is positive.
On the other hand, since we are looking for $2\pi$-periodic solutions, we could
work on any interval of length $2\pi$, for instance $\mathopen{[}\pi,3\pi\mathclose{]}$. This is equivalent to consider the
periodic boundary conditions on $\mathopen{[}0,2\pi\mathclose{]}$ for the equation
\begin{equation*}
v'' + (-k +\cos(x-\pi)) g(v) = 0.
\end{equation*}
In this latter case, the weight is negative at the endpoints $x= 0$ and $x=2\pi$ and there is only one
interval of non-negativity, so that we have to consider the eigenvalue problem \eqref{eq-eig} only on
$I_{1} = \mathopen{[}\pi - \arccos k, \pi + \arccos k\mathclose{]}$. In this way we can produce a better lower bound for $g_{\infty}$
by studying an equivalent problem.

The same remarks as those just made above apply in any subsequent variant of Theorem~\ref{MainTheorem}, for instance, we
can refine the choice of the constants $\lambda_{1}^{i}$ in Corollary~\ref{cor-3.4} below.
$\hfill\lhd$
\end{remark}

The following corollaries are straightforward consequences of Theorem~\ref{MainTheorem}.

\begin{corollary}\label{cor-3.1}
Let $g(s)$ and $a(x)$ satisfy $(g_{1})$ and $(a_{1})$, $(a_{2})$, respectively.
Suppose also that $g(s)$ is regularly oscillating at zero and satisfies
\begin{equation*}
\lim_{s\to 0^{+}}\dfrac{g(s)}{s}=0 \quad \text{ and } \quad \lim_{s\to +\infty}\dfrac{g(s)}{s}=+\infty.
\end{equation*}
Then problem $({\mathscr{P}})$ has at least one positive solution.
\end{corollary}

\begin{corollary}\label{cor-3.2}
Let $g(s)$ and $a(x)$ satisfy $(g_{1})$ and $(a_{1})$, $(a_{2})$, respectively.
Suppose also that $g(s)$ is regularly oscillating at zero and satisfies
\begin{equation*}
\lim_{s\to 0^{+}}\dfrac{g(s)}{s}=0 \quad \text{ and } \quad g_{\infty} > 0.
\end{equation*}
Then there exists $\nu^{*}>0$ such that the boundary value problem
$({\mathscr{P}}_{\nu})$
has at least one positive solution for each $\nu > \nu^{*}$.
\end{corollary}

The following consequence of Theorem~\ref{MainTheorem} provides a necessary and sufficient condition for the existence of positive solutions
to problem $({\mathscr{P}})$ when $g(s)=s^{\gamma}$ for $\gamma>1$. It can be viewed as a version of \cite[Theorem~1]{BeCaDoNi-1995} for the
periodic case (in \cite{BeCaDoNi-1995} the authors already obtained the same result for the Neumann problem for PDEs).

\begin{corollary}\label{cor-3.3}
The superlinear boundary value problem
\begin{equation*}
\begin{cases}
\, u'' + a(x)u^{\gamma}=0, \quad \gamma > 1, \\
\, {\mathscr{B}}(u,u') = \underline{0},
\end{cases}
\end{equation*}
with a weight function $a(x)$ satisfying $(a_{1})$,
has a positive solution if and only if the average condition $(a_{2})$ holds.
\end{corollary}

\begin{proof}
The necessary part of the statement is a consequence of the fact that $g(s):= s^{\gamma}$, for $\gamma > 1$,
has a positive derivative on ${\mathbb{R}}^{+}_{0}$ (see also the Introduction). For the sufficient part we
apply Corollary~\ref{cor-3.1}, observing that $g(s)$ is regularly oscillating at zero.
\end{proof}

\medskip

As one can clearly notice from the proof, the condition $g_{\infty} > \max \lambda_{1}^{i}$ in Theorem~\ref{MainTheorem}
is required in order to obtain suitable a priori bounds $R_{I_{i}}$ for the maximum of the solutions on each of the intervals
$I_{i}$, where $a\geq 0$ and $a\not\equiv 0$. Hence we can choose a constant $R$ satisfying \eqref{eq-R} and have
$(H_{R})$ verified. As observed in \cite{FeZa-pre2014}, if $g(s)/s$ is bounded (for $s$ large), it is sufficient to obtain an a priori
bound \textit{only on one} of the intervals $I_{i}$ and the existence of a global upper bound follows from standard ODEs arguments
related to the classical Gronwall's inequality. Similarly, also in the present situation, following the proof of
\cite[Theorem~3.2]{FeZa-pre2014}, we can obtain the next result.

\begin{corollary}\label{cor-3.4}
Let $g(s)$ and $a(x)$ satisfy $(g_{1})$ and $(a_{1})$, $(a_{2})$, respectively.
Suppose also that $g(s)$ is regularly oscillating at zero and satisfies
\begin{equation*}
\lim_{s\to 0^{+}}\dfrac{g(s)}{s}=0, \quad
g_{\infty} > \min_{i=1,\ldots,m} \lambda_{1}^{i}
\quad \text{ and } \quad
\limsup_{s\to +\infty} \dfrac{g(s)}{s} < +\infty.
\end{equation*}
Then problem $({\mathscr{P}})$ has at least one positive solution.
\end{corollary}

Actually, this result could be even improved with respect to assumption $(a_{1})$, in the sense that it would be sufficient only to find an interval
$J\subseteq \mathopen{[}0,T\mathclose{]}$ where $a\geq 0$ and $a\not\equiv 0$ and then
we can ignore completely the behavior of $a(x)$ on $\mathopen{[}0,T\mathclose{]}\setminus J$. If we know that $g_{\infty}$ is greater than the first
eigenvalue of the Dirichlet problem in $J$ (i.e.~$\varphi'' + \lambda a(x)\varphi = 0$, $\varphi|_{\partial J}=0$),
we get the upper bound $R_{J}$ as in Lemma~\ref{lemma_RJ} and hence a global upper bound via Gronwall's inequality.
Thus we can prove the following corollary which combines Corollary~\ref{cor-3.2} with Corollary~\ref{cor-3.4}.

\begin{corollary}\label{cor-3.5}
Assume there exists an interval $J\subseteq \mathopen{[}0,T\mathclose{]}$ where $a(x) \geq 0$ for a.e.~$x\in J$ and also
\begin{equation*}
\int_{0}^{T} a(x) ~\!dx < 0 < \int_{J} a(x) ~\!dx.
\end{equation*}
Suppose that $g(s)$ is regularly oscillating at zero, satisfying $(g_{1})$ and such that
\begin{equation*}
\lim_{s\to 0^{+}}\dfrac{g(s)}{s}=0 \quad \text{ and } \quad 0  < \liminf_{s\to +\infty}
\dfrac{g(s)}{s} \leq \limsup_{s\to +\infty} \dfrac{g(s)}{s} < +\infty.
\end{equation*}
Then there exists $\nu^{*}>0$ such that the boundary value problem
$({\mathscr{P}}_{\nu})$
has at least one positive solution for each $\nu > \nu^{*}$.
\end{corollary}

\begin{remark}[\textit{A comment on the regularly oscillating condition}]\label{rem-3.2}
The condition of regularly oscillation at zero required on $g(s)$ is useful in order to
conclude the verification of $(H_{r})$ in Theorem~\ref{MainTheorem}.
Nevertheless, there is a disadvantage in assuming such a condition, as it does not allow to consider functions as
\begin{equation}\label{eq-exp}
g (s) = s^{\gamma} \exp(-1/s) \; \text{ for } \; s > 0 \;\; (\gamma > 1), \quad g(0) = 0,
\end{equation}
which are not regularly oscillating at zero.

With this respect we observe that, from a careful reading of the first part of the proof of Theorem~\ref{MainTheorem},
the key point is to show that $g(r_{n}\omega_{n})/g(r_{n}) \to 1$, for $r_{n}\to 0^{+}$ and $\omega_{n} = v(x_{n}) \to 1$
(for a suitable choice of $x_{n}\in \mathopen{[}0,T\mathclose{]}$).
In our proof, the sequence $\omega_{n}$ is not an arbitrary sequence tending to $1$, since $0 < \omega_{n} <1$,
and, moreover, from \eqref{eq-3.7} we can easily provide the estimate
\begin{equation*}
0 \leq  1 - \omega_{n} \leq C \sup_{0 < s \leq r_{n}} \dfrac{g(s)}{s}
\end{equation*}
(for $C > 0$ a suitable constant independent on $r_{n}$ and $\omega_{n}$).
Therefore, the faster ${g(r_{n})}/{r_{n}}$ tends to zero, the more $\omega_{n}$ tends to one.

Using this observation, we can apply our result also to some not regularly oscillating functions $g(s)$, provided that they
tend to zero sufficiently fast. In this manner, for instance, Corollary~\ref{cor-3.1} holds also for a
function $g(s)$ as in \eqref{eq-exp} (the easy verification is omitted).

Another way to avoid the hypothesis of regular oscillation at the origin is described in Theorem~\ref{MainTheorem2}.
$\hfill\lhd$
\end{remark}

\medskip

A variant of Theorem~\ref{MainTheorem} is the following result. Basically, we replace
the computations for the verification of $(H_{r})$ given in the proof of Theorem~\ref{MainTheorem}
with a different argument which is essentially inspired by the approach in~\cite{BoZa-2013}.

\begin{theorem}\label{MainTheorem2}
Let $g \colon {\mathbb{R}}^{+}\to{\mathbb{R}}^{+}$ be a continuous function satisfying $(g_{1})$ and $(g_{2})$.
Let $a\colon \mathopen{[}0,T\mathclose{]}\to {\mathbb{R}}$ be a measurable function satisfying $(a_{1})$ and $(a_{2})$.
Suppose also that $g(s)$ is continuously differentiable on a right neighborhood $\mathopen{[}0,\varepsilon_{0}\mathclose{[}$
of $s=0$.
Then problem $({\mathscr{P}})$ has at least one positive solution.
\end{theorem}

\begin{proof}
As in the proof of Theorem~\ref{MainTheorem}, we enter in the setting of Theorem~\ref{th-2.1} by defining
\begin{equation*}
f(x,s,\xi) = f(x,s) := a(x)g(s).
\end{equation*}

\medskip

\noindent
\textit{Verification of $(H_{r})$.}
First of all, we observe that $(g_{1})$ and $(a_{2})$ imply that
\begin{equation}\label{eq-fs2}
\int_{0}^{T} f(x,s,0)~\!dx < 0, \quad \forall \, s > 0.
\end{equation}
We claim that there exists $r_{0}\in \mathopen{]}0,\varepsilon_{0}\mathclose{[}$
such that for all $0<r\leq r_{0}$ and
for all $\vartheta\in\mathopen{]}0,1\mathclose{]}$ there are no solutions $u(x)$ of \eqref{eq-2.9}
such that $u(x)>0$ on $\mathopen{[}0,T\mathclose{]}$ and $\|u\|_{\infty}=r$.

By contradiction, suppose the claim is not true. Then for all
$n\in{\mathbb{N}}$ there exist $0<r_{n} <1/n$,
$\vartheta_{n}\in\mathopen{]}0,1\mathclose{]}$ and $u_{n}(x)$
solution of \eqref{eq-vartheta-n} such that $u_{n}(x)>0$ on
$\mathopen{[}0,T\mathclose{]}$ and $\|u_{n}\|_{\infty} = r_{n}$. By the first condition in $(g_{2})$, note
also that
\begin{equation*}
\lim_{n\to\infty} g'(u_{n}(x)) = 0, \quad \text{uniformly on } \mathopen{[}0,T\mathclose{]}.
\end{equation*}
Using the identity
\begin{equation*}
\dfrac{u''}{g(u)} = \dfrac{d}{dx}\biggl(\dfrac{u'}{g(u)}\biggr) + g'(u) \biggl( \dfrac{u'}{g(u)} \biggr)^{\!2},
\end{equation*}
for $u = u_{n}$, and setting
\begin{equation*}
z_{n}(x):=  \dfrac{u_{n}'(x)}{g(u_{n}(x))}, \quad x \in \mathopen{[}0,T\mathclose{]},
\end{equation*}
we obtain the following relation
\begin{equation}\label{eqz-eq}
z_{n}'(x) + g'(u_{n}(x)) z_{n}^{2}(x) = - \vartheta_{n}a(x), \quad \text{for a.e. } x\in \mathopen{[}0,T\mathclose{]}.
\end{equation}
The boundary conditions (of Neumann or periodic type) on $u_{n}(x)$ imply that
\begin{equation}\label{eqz-BC}
z_{n}(0) = z_{n}(T) \quad \text{ and } \quad \exists \, t^{*}_{n}\in \mathopen{[}0,T\mathclose{]} \colon z_{n}(t^{*}_{n}) = 0
\end{equation}
(obviously, we can take $t^{*}_{n}= 0$ in the case of the Neumann boundary conditions, while the existence of such a point
in the periodic case follows from Rolle's theorem).

We fix a positive constant $M > \|a\|_{L^{1}}$ and then a constant $\delta$ with
\begin{equation}\label{eq-delta}
0 < \delta < \dfrac{M - \|a\|_{L^{1}}}{T M^{2}}.
\end{equation}
By the continuity of $g'(s)$ on $\mathopen{[}0,\varepsilon_{0}\mathclose{[}$ and $g'(0) = 0$
(which corresponds to the first condition in $(g_{2})$),
we find $\varepsilon\in \mathopen{]}0,\varepsilon_{0}\mathclose{[}$ such that
\begin{equation*}
|g'(s)|\leq \delta, \quad \forall \, 0 \leq s \leq \varepsilon.
\end{equation*}
Let $n > 1/\varepsilon$. In this case, we have that $0 < u_{n}(x) < \varepsilon$ on $\mathopen{[}0,T\mathclose{]}$ and we claim that
\begin{equation}\label{eqz-ineq}
\|z_{n}\|_{\infty} \leq \vartheta_{n} M.
\end{equation}
Indeed, if by contradiction we suppose that \eqref{eqz-ineq} is not true (for some $n > 1/\varepsilon$), then, using the fact that $z_{n}(x)$
vanishes at some point $t^{*}_{n}$ of $\mathopen{[}0,T\mathclose{]}$, we can find a maximal interval $J_{n}$
of the form $\mathopen{[}t^{*}_{n},\tau_{n}\mathclose{]}$ or $\mathopen{[}\tau_{n},t^{*}_{n}\mathclose{]}$
such that $|z_{n}(x)| \leq \vartheta_{n} M$ for all $x\in J_{n}$ and $|z_{n}(x)| > \vartheta_{n} M$ for some $x\not\in J_{n}$,
or, more precisely, with $\tau_{n} < x \leq T$ or $0 \leq x < \tau_{n}$, respectively.
By the maximality of the interval $J_{n}$, we also know that $|z_{n}(\tau_{n})| = \vartheta_{n} M$.

Integrating \eqref{eqz-eq} on $J_{n}$ and passing to the absolute value, we obtain
\begin{equation*}
\begin{aligned}
\vartheta_{n} M &= |z_{n}(\tau_{n})| = |z_{n}(\tau_{n}) - z_{n}(t^{*}_{n})| \\
& \leq  \Bigl{|} \int_{J_{n}} g '(u_{n}(x)) z_{n}^{2}(x) ~\!dx \Bigr{|} + \vartheta_{n}\|a\|_{L^{1}}\\
& \leq \delta \vartheta_{n}^{2} M^{2} |\tau_{n} - t^{*}_{n}| + \vartheta_{n}\|a\|_{L^{1}} \\
& \leq \vartheta_{n} \bigl(\delta  M^{2} T + \|a\|_{L^{1}} \bigr)
\end{aligned}
\end{equation*}
(recall that $0 < \vartheta_{n} \leq 1$). Dividing the above inequality by $\vartheta_{n} > 0$, we find a contradiction
with the choice of $\delta$ in \eqref{eq-delta}. In this manner, we have verified that \eqref{eqz-ineq} is true.

Now, integrating \eqref{eqz-eq} on $\mathopen{[}0,T\mathclose{]}$, recalling that $z_{n}(T) - z_{n}(0) = 0$ (according to \eqref{eqz-BC})
and using \eqref{eqz-ineq}, we obtain that
\begin{equation*}
\begin{aligned}
- \vartheta_{n} \int_{0}^{T} a(x)~\!dx  &= \int_{0}^{T} g '(u_{n}(x)) z_{n}^{2}(x) ~\!dx\\
&\leq T \vartheta_{n}^{2} M^{2} \max_{0\leq s \leq r_{n}}|g'(s)|\\
&\leq \vartheta_{n} T M^{2} \max_{0\leq s \leq r_{n}}|g'(s)|
\end{aligned}
\end{equation*}
holds for every $n > 1/\varepsilon$. From this,
\begin{equation}\label{eq-3.18}
0<-\bar{a}\leq M^{2} \max_{0\leq s \leq r_{n}}|g'(s)|
\end{equation}
follows. Using the continuity of $g'(s)$ at $s=0^{+}$, we get a contradiction, as $n\to \infty$.

The claim is thus proved and, recalling also \eqref{eq-fs2}, we have that $(H_{r})$ holds
for any $r\in \mathopen{]}0,r_{0}\mathclose{]}$.

\medskip

\noindent
\textit{Verification of $(H_{R})$.} This has been already checked  in the second part of the proof of Theorem~\ref{MainTheorem}.
No change is needed.

\medskip

\noindent
As last step, we conclude exactly as in the proof of Theorem~\ref{MainTheorem},
via Theorem~\ref{th-2.1}.
\end{proof}

{} From Theorem~\ref{MainTheorem2}, we can derive the same corollaries as above in which the
condition of regularly oscillation at zero of $g(s)$ is systematically replaced by the smoothness of $g(s)$
on a right neighborhood $\mathopen{[}0,\varepsilon_{0}\mathclose{[}$ of zero. In particular, an obvious
improvement of Corollary~\ref{cor-3.3} is the following.

\begin{corollary}\label{cor-3.6}
Let $g \colon {\mathbb{R}}^{+}\to{\mathbb{R}}^{+}$ be a continuously differentiable function satisfying $(g_{1})$, such that
$g'(s) > 0$ for all $s > 0$ and
\begin{equation*}
\lim_{s\to 0^{+}}\dfrac{g(s)}{s}=0, \qquad \lim_{s\to +\infty}\dfrac{g(s)}{s}=+\infty.
\end{equation*}
Let $a\colon \mathopen{[}0,T\mathclose{]}\to {\mathbb{R}}$ be a measurable function satisfying $(a_{1})$.
Then problem $({\mathscr{P}})$ has at least one positive solution if and only if $(a_{2})$ holds.
\end{corollary}

{} For the Neumann problem this result improves \cite[\S~3, Corollary~1]{BoZa-2013} to a more general class of weight functions $a(x)$.
It also  extends such a result to the periodic case.

\begin{remark}[\textit{A comparison between different conditions at zero}]\label{rem-3.3}
In Theorem~\ref{MainTheorem} and Theorem~\ref{MainTheorem2} we have two different conditions that
are required on $g(s)$ as $s\to 0^{+}$. It can be interesting to provide examples in which one of
the two results applies, while for the other the conditions on $g(s)$ are not fulfilled.
For this discussion, we confine ourselves only to the behavior of $g(s)$ on a right
neighborhood $\mathopen{[}0,1\mathclose{]}$ of zero and we do not care about $a(x)$ or the behavior of $g(s)$ as $s\to +\infty$.

Take any function $\sigma \colon \mathopen{[}0,1\mathclose{]}\to {\mathbb{R}}^{+}_{0}$ which is continuous but not differentiable
(for instance, one could even choose a nowhere differentiable function of Weierstrass type) and define
\begin{equation*}
g(s) = \sigma(s) s^{\gamma}, \quad \gamma > 1.
\end{equation*}
Such a function $g(s)$ is regularly oscillating at zero (note that $\sigma(0) > 0$) and it fits for
Theorem~\ref{MainTheorem}, but it is not suitable for Theorem~\ref{MainTheorem2}.

As second example, we consider a function as
\begin{equation*}
g(s) =  s^{\gamma} \sin^{2}(1/s) + s^{\beta}, \; \text{ for } \; s \in \mathopen{]}0,1\mathclose{]} \;\; (\beta > \gamma > 2), \quad g(0) = 0.
\end{equation*}
Such a function $g(s)$ is continuously differentiable on $\mathopen{[}0,1\mathclose{]}$ and it fits for
Theorem~\ref{MainTheorem2}, but it is not suitable for Theorem~\ref{MainTheorem} since $g(s)$ is not regularly oscillating at zero.

Both the above examples can be easily generalized in order to construct broad classes of nonlinearities
where only one of the two theorems applies.
$\hfill\lhd$
\end{remark}

\medskip

We end this section by presenting a variant of Corollary~\ref{cor-3.5} in the smooth case and
observing that the argument employed in the proof of Theorem~\ref{MainTheorem2} can be used to provide a
\textit{nonexistence result} for positive solutions when $g(s)$ is smooth on ${\mathbb{R}}^{+}$ and with sufficiently small derivative.

\begin{corollary}\label{cor-3.7}
Assume there exists an interval $J\subseteq \mathopen{[}0,T\mathclose{]}$ where $a(x) \geq 0$ for a.e.~$x\in J$ and also
\begin{equation*}
\int_{0}^{T} a(x) ~\!dx < 0 < \int_{J} a(x) ~\!dx.
\end{equation*}
Suppose also that $g(s)$ is a continuously differentiable function satisfying $(g_{1})$ and such that
\begin{equation*}
g'(0) = 0 \quad \text{ and } \quad 0  < \liminf_{s\to +\infty} g'(s) \leq \limsup_{s\to +\infty} g'(s) < +\infty.
\end{equation*}
Then there exists $\nu^{*}>0$ such that the boundary value problem
$({\mathscr{P}}_{\nu})$
has at least one positive solution for each $\nu > \nu^{*}$.
\end{corollary}

Clearly, Theorem~\ref{th-1.3} in the Introduction is a direct consequence of this result, using the
generalized de l'H\^{o}pital's rule:
\begin{equation*}
\liminf_{s\to +\infty} g'(s) \leq \liminf_{s\to +\infty} \dfrac{g(s)}{s} \leq \limsup_{s\to +\infty} \dfrac{g(s)}{s} \leq \limsup_{s\to +\infty} g'(s).
\end{equation*}

\begin{proposition}\label{prop-3.1}
Let $g \colon {\mathbb{R}}^{+}_{0}\to{\mathbb{R}}^{+}_{0}$ be a continuously differentiable function
with bounded derivative on ${\mathbb{R}}^{+}_{0}$. Let $a \in L^{1}(\mathopen{[}0,T\mathclose{]})$ satisfy condition $(a_{2})$.
Then there exists ${\nu}_{*} > 0$ such that
the boundary value problem
$({\mathscr{P}}_{\nu})$
has no positive solutions for each $0 <\nu < {\nu}_{*}$.
\end{proposition}

\begin{proof}
The proof follows substantially the same argument employed in the proof of Theorem~\ref{MainTheorem2} from \eqref{eqz-eq} to \eqref{eq-3.18}.

We fix two positive constants $M$ and $D$ such that
\begin{equation*}
M > \|a\|_{L^{1}} \quad \text{ and } \quad |g'(s)| \leq D, \;\; \forall \, s > 0,
\end{equation*}
(recall that, by assumption, $g(s)$ has bounded derivative on ${\mathbb{R}}^{+}_{0}$) and define
\begin{equation*}
\nu_{*} := \min \biggl{\{} \dfrac{M-\|a\|_{L^{1}}}{D M^{2} T} , \dfrac{-\bar{a}}{D M^{2}}\biggr{\}}.
\end{equation*}
We shall prove that for $0 < \nu < \nu_{*}$ problem $({\mathscr{P}}_{\nu})$ has no positive solution.

Let us suppose by contradiction that $u(x)>0$ for all $x\in \mathopen{[}0,T\mathclose{]}$ is a solution of problem $({\mathscr{P}}_{\nu})$.
Setting $z(x):= u'(x)/\nu g(u(x))$, we find
\begin{equation}\label{eqz-eq2}
z'(x) + \nu g'(u(x))z^{2}(x)= - a(x).
\end{equation}
As a consequence of the boundary conditions, we also have that $z(0) = z(T)$
and there exists $t^{*}\in \mathopen{[}0,T\mathclose{]}$
(with $t^{*}$ depending on the solution $u(x)$) with $z(t^{*}) = 0$.

First of all, we claim that
\begin{equation}\label{eqz-ineq2}
\|z\|_{\infty} \leq M.
\end{equation}
Indeed,
if by contradiction we suppose that \eqref{eqz-ineq2} is not true, then using the fact that $z(x)$
vanishes at some point of $\mathopen{[}0,T\mathclose{]}$, we can find a maximal interval $J$
of the form $\mathopen{[}t^{*},\tau\mathclose{]}$ or $\mathopen{[}\tau,t^{*}\mathclose{]}$
such that $|z(x)| \leq M$ for all $x\in J$ and $|z(x)| > M$ for some $x\not\in J$.
By the maximality of the interval $J$, we also know that $|z(\tau)| = M$.
Integrating \eqref{eqz-eq2} on $J$ and passing to the absolute value, we obtain
\begin{equation*}
\begin{aligned}
M &= |z(\tau) - z(t^{*})| \leq  \biggl{|} \int_{J} \nu g '(u(x)) z^{2}(x) ~\!dx \biggr{|} + \|a\|_{L^{1}} \\
& \leq \nu D M^{2} T + \|a\|_{L^{1}} < M,
\end{aligned}
\end{equation*}
a contradiction. In this manner, we have verified that \eqref{eqz-ineq2} is true.

Now, integrating \eqref{eqz-eq2} on $\mathopen{[}0,T\mathclose{]}$, recalling that $z(T) - z(0) = 0$
and using \eqref{eqz-ineq2}, we reach
\begin{equation*}
- \bar{a} = - \dfrac{1}{T} \int_{0}^{T} a(x)~\!dx  = \dfrac{1}{T} \int_{0}^{T} \nu g '(u(x)) z^{2}(x) ~\!dx
\leq \nu D M^{2}  <  - \bar{a},
\end{equation*}
a contradiction. This concludes the proof.
\end{proof}

\begin{remark}\label{rem-3.4}
The same proof as above works to prove the nonexistence of solution to problem $({\mathscr{P}}_{\nu})$
with range in a given open interval $\mathopen{]}\alpha,\beta\mathclose{[}$, for
$g \colon \mathopen{]}\alpha,\beta\mathclose{[} \to {\mathbb{R}}^{+}_{0}$ a smooth function with bounded derivative.
In some recent papers (see \cite{BoGa-pre2014,BoZa-2012}) similar nonexistence results have been obtained
under different conditions on the function $g(s)$.
$\hfill\lhd$
\end{remark}

\section{More general examples and applications}\label{section4}

In Section~\ref{section3} we have applied our abstract result Theorem~\ref{th-2.1}, which deals with
a general second order equation of the form
\begin{equation*}
u'' + f(x,u,u') = 0,
\end{equation*}
to the simpler case given by
\begin{equation*}
u'' + a(x) g(u) = 0.
\end{equation*}
In this section, we show how our result can be extended to a broader class of equations.
Up to this point, by ${\mathscr{B}}(u,u') = \underline{0}$, we have considered together the two different boundary conditions.
Now we present two different applications, one for the Neumann problem and another for periodic solutions.

{} For simplicity in the exposition, in our applications we will suppose that the weight function
is \textit{continuous}, in order to obtain classical positive solutions.

\subsection{The Neumann problem: radially symmetric solutions}\label{section4.1}

Let $\|\cdot\|$ be the Euclidean norm in ${\mathbb{R}}^{N}$ (for $N \geq 2$) and let
\begin{equation*}
\Omega:= B(0,R_{2})\setminus B[0,R_{1}] = \{x\in {\mathbb{R}}^{N} \colon R_{1} < \|x\| < R_{2}\}
\end{equation*}
be an open annular domain, with $0 < R_{1} < R_{2}$.
Let $q \colon \overline{\Omega}\to {\mathbb{R}}$ be a continuous function which is radially symmetric, namely
there exists a continuous scalar function ${\mathcal{Q}} \colon \mathopen{[}R_{1},R_{2}\mathclose{]}\to {\mathbb{R}}$
such that
\begin{equation*}
q(x) = {\mathcal{Q}}(\|x\|), \quad \forall \, x\in \overline{\Omega}.
\end{equation*}
In this section we consider the Neumann boundary value problem
\begin{equation}\label{eq-pde-rad}
\begin{cases}
\, -\Delta \,u = q(x)\,g(u) & \text{ in } \Omega \\
\, \dfrac{\partial u}{\partial {\bf n}} = 0 & \text{ on } \partial\Omega
\end{cases}
\end{equation}
and we are interested in the existence of radially symmetric positive solutions of \eqref{eq-pde-rad},
namely classical solutions such that $u(x) > 0$ for all $x\in \Omega$ and also $u(x') = u(x'')$ whenever
$\|x'\| = \|x''\|$.

Since we look for radially symmetric solutions of \eqref{eq-pde-rad}, our study can be reduced to the search of positive solutions
of the Neumann boundary value problem
\begin{equation}\label{eq-rad}
w''(r) + \dfrac{N-1}{r}w'(r) + {\mathcal{Q}}(r) g(w(r)) = 0, \quad w'(R_{1}) = w'(R_{2}) = 0.
\end{equation}
Indeed, if $w(r)$ is a solution of \eqref{eq-rad}, then $u(x):= w(\|x\|)$ is a solution of \eqref{eq-pde-rad}.
Using the standard change of variable
\begin{equation*}
t = h(r):= \int_{R_{1}}^{r} \xi^{1-N} ~\!d\xi
\end{equation*}
and defining
\begin{equation*}
T:= \int_{R_{1}}^{R_{2}} \xi^{1-N} ~\!d\xi, \quad r(t):= h^{-1}(t) \quad \text{and} \quad v(t)=w(r(t)),
\end{equation*}
we transform \eqref{eq-rad} into the equivalent problem
\begin{equation}\label{eq-rad1}
v'' +  a(t) g(v) = 0, \quad v'(0) = v'(T) = 0,
\end{equation}
with
\begin{equation*}
a(t):= r(t)^{2(N-1)}{\mathcal{Q}}(r(t)).
\end{equation*}
Consequently, the Neumann boundary value problem \eqref{eq-rad1} is of the same form of $({\mathscr{P}})$
and we can apply the results of Section~\ref{section3}.

Since $r(t)^{2(N-1)} > 0$ on $\mathopen{[}0,T\mathclose{]}$, condition $(a_{1})$ is satisfied provided that
a similar condition holds for ${\mathcal{Q}}(r)$ on $\mathopen{[}R_{1},R_{2}\mathclose{]}$.
Accordingly, we assume
\begin{itemize}
\item[$(q_{1})$]
\textit{there exist $m \geq 1$ intervals $J_{1},\ldots,J_{m}$, closed and pairwise disjoint, such that
such that
\begin{equation*}
\begin{aligned}
& {\mathcal{Q}}(r)\geq 0, \;\; \text{for every } \; r\in J_{i}, \text{ with }\; \max_{r\in J_{i}} {\mathcal{Q}}(r) > 0 \quad (i=1,\ldots,m);\\
& {\mathcal{Q}}(r)\leq 0, \;\; \text{for every } \; r\in \mathopen{[}R_{1},R_{2}\mathclose{]}\setminus\bigcup_{i=1}^{m} J_{i}.
\end{aligned}
\end{equation*}}
\end{itemize}
Condition $(a_{2})$ reads as
\begin{equation*}
0 > \int_{0}^{T}r(t)^{2(N-1)}{\mathcal{Q}}(r(t))~\!dt = \int_{R_{1}}^{R_{2}}r^{N-1}{\mathcal{Q}}(r)~\!dr.
\end{equation*}
Up to a multiplicative constant, the latter integral is the integral of $q(x)$ on $\Omega$,
using the change of variable formula for radially symmetric functions (cf.~\cite{Fo-1984}). Thus, $a(t)$ satisfies $(a_{2})$ if and only if
\begin{equation*}
\int_{\Omega}^{} q(x)~\!dx < 0.
\leqno{\hspace*{2.2pt}(q_{2})}
\end{equation*}

The following theorems are easy corollaries of the results presented in Section~\ref{section3}.

\begin{theorem}\label{th-radial-1}
Let $g \colon {\mathbb{R}}^{+}\to{\mathbb{R}}^{+}$ be a continuous function, regularly oscillating at zero
and satisfying $(g_{1})$. Assume
\begin{equation*}
\lim_{s\to 0^{+}}\dfrac{g(s)}{s}=0 \quad \text{ and } \quad \lim_{s\to +\infty}\dfrac{g(s)}{s}=+\infty.
\end{equation*}
Let $q(x) = {\mathcal{Q}}(\|x\|)$ be a continuous radially symmetric function satisfying $(q_{1})$ and $(q_{2})$.
Then problem \eqref{eq-pde-rad} has at least one positive radially symmetric solution.
\end{theorem}

\begin{theorem}\label{th-radial-2}
Let $g \colon {\mathbb{R}}^{+}\to{\mathbb{R}}^{+}$ be a continuous function, regularly oscillating at zero
and satisfying $(g_{1})$. Assume
\begin{equation*}
\lim_{s\to 0^{+}}\dfrac{g(s)}{s}=0 \quad \text{ and } \quad \liminf_{s\to +\infty}\dfrac{g(s)}{s} > 0.
\end{equation*}
Let $q(x) = {\mathcal{Q}}(\|x\|)$ be a continuous radially symmetric function satisfying $(q_{1})$ and $(q_{2})$.
Then there exists $\nu^{*} > 0$ such that problem
\begin{equation}\label{eq-pde-radnu}
\begin{cases}
\, -\Delta \,u = \nu \, q(x)\,g(u) & \text{ in } \Omega \\
\, \dfrac{\partial u}{\partial {\bf n}} = 0 & \text{ on } \partial\Omega
\end{cases}
\end{equation}
has at least one positive radially symmetric solution for each $\nu > \nu^{*}$.
\end{theorem}

Clearly, Theorem~\ref{th-radial-1} and Theorem~\ref{th-radial-2} correspond to Corollary~\ref{cor-3.1} and Corollary~\ref{cor-3.2}, respectively.
The next result follows from the same argument that led to Corollary~\ref{cor-3.5}.

\begin{theorem}\label{th-radial-3}
Let $g \colon {\mathbb{R}}^{+}\to{\mathbb{R}}^{+}$ be a continuous function, regularly oscillating at zero
and satisfying $(g_{1})$. Assume
\begin{equation*}
\lim_{s\to 0^{+}}\dfrac{g(s)}{s}=0 \quad \text{ and } \quad 0 < \liminf_{s\to +\infty}\dfrac{g(s)}{s} \leq \limsup_{s\to +\infty}\dfrac{g(s)}{s}
< +\infty.
\end{equation*}
Let $q(x) = {\mathcal{Q}}(\|x\|)$ be a continuous radially symmetric function satisfying $(q_{2})$
and such that $q(x_{0}) > 0$ for some $x_{0}\in \Omega$.
Then there exists $\nu^{*} > 0$ such that problem \eqref{eq-pde-radnu}
has at least one positive radially symmetric solution for each $\nu > \nu^{*}$.
\end{theorem}

Note that, with respect to Theorem~\ref{th-radial-2}, in the above result we do not assume condition $(q_{1})$
on the weight function. In this manner, we can consider
functions ${\mathcal{Q}}(r)$ with infinitely many changes of sign in $\mathopen[R_{1},R_{2}\mathclose]$.

All the above three theorems can be stated in a version where the regularly oscillating assumption at zero is replaced
with the hypothesis that $g(s)$ is continuously differentiable on a right neighborhood of zero,
according to Theorem~\ref{MainTheorem2}. For instance, the
corresponding version of Theorem~\ref{th-radial-1} reads as follows.

\begin{theorem}\label{th-radial-4}
Let $g \colon {\mathbb{R}}^{+}\to{\mathbb{R}}^{+}$ be a continuous function satisfying $(g_{1})$ and such that
$g(s)$ is continuously differentiable on a right neighborhood of $s=0$. Assume
\begin{equation*}
g'(0)=0 \quad \text{ and } \quad \lim_{s\to +\infty}\dfrac{g(s)}{s}=+\infty.
\end{equation*}
Let $q(x) = {\mathcal{Q}}(\|x\|)$ be a continuous radially symmetric function satisfying $(q_{1})$ and $(q_{2})$.
Then problem \eqref{eq-pde-rad} has at least one positive radially symmetric solution.
\end{theorem}

If we also suppose that $g(s)$ is continuously differentiable on ${\mathbb{R}}^{+}_{0}$ with
$g'(s) > 0$ for all $s > 0$, then condition $(q_{2})$ is also necessary for the existence of a positive solution.
On the other hand, as observed in the Introduction, also the fact that the weight function must change its sign is
necessary for the existence of solutions. With this respect, the following corollary can be derived from the smooth version of
Theorem~\ref{th-radial-3} (see also Corollary~\ref{cor-3.7}).

\begin{corollary}\label{cor-radial-1}
Let $g \colon {\mathbb{R}}^{+}\to{\mathbb{R}}^{+}$ be a continuously differentiable function such that
$g'(s) > 0$ for all $s > 0$. Assume
\begin{equation*}
g(0) = g'(0) = 0 \quad \text{ and } \quad g'(+\infty) = \ell > 0.
\end{equation*}
Let $q(x) = {\mathcal{Q}}(\|x\|)$ be a continuous radially symmetric function.
Then there exists $\nu^{*} > 0$ such that problem \eqref{eq-pde-radnu}
has at least one positive radially symmetric solution for each $\nu > \nu^{*}$
if and only if
\begin{equation*}
q^{+}(x)\not\equiv 0 \quad \text{ and } \quad \int_{\Omega} q(x)~\!dx < 0.
\end{equation*}
\end{corollary}

Note also that, under the assumptions of Corollary~\ref{cor-radial-1} there is also a constant $\nu_{*} > 0$ such that
for each $0 < \nu < \nu_{*}$ problem \eqref{eq-pde-radnu} has no positive radial solutions (cf.~Proposition~\ref{prop-3.1}).

\medskip

Possible examples of functions satisfying the above conditions are
\begin{equation*}
g(s) = K s \arctan(s^{\gamma-1}) \quad \text{for } \; s \geq 0 \qquad (\gamma > 1, \; K > 0)
\end{equation*}
and
\begin{equation*}
g(s) = K \dfrac{s^{\gamma}}{s^{\gamma-1} + M} \quad \text{for } \; s \geq 0 \qquad (\gamma > 1, \; K,M > 0).
\end{equation*}

\begin{remark}\label{rem-4.1}
In \cite{BeCaDoNi-1994}, Berestycki, Capuzzo-Dolcetta and Nirenberg obtained an existence result of
positive solutions for the Neumann problem \eqref{eq-pde-rad} in the superlinear indefinite case
for $\Omega$ a bounded domain with smooth boundary. In \cite[Theorem~3]{BeCaDoNi-1994} the main assumptions require
that $g(s)$ has a precise power-like growth at infinity,
that is $g(s)/s^{p} \to l >0$ (as $s\to +\infty$) for some $p \in \mathopen]1, (N+2)/(N-1)\mathclose[$,
and that $\nabla q(x)$ does not vanish on the points of
\begin{equation*}
\Gamma:=\overline{\{x\in \Omega \colon q(x) > 0\}} \cap \overline{\{x\in \Omega \colon q(x) < 0\}}\subseteq \Omega.
\end{equation*}
Our setting is much more simplified as we consider an annular domain and a radially symmetric weight function.
On the other hand, our growth condition at infinity is more general (allowing a nonlinearity which is not necessarily of power-like type)
and, moreover, no condition on the zeros of $q(x)$ is required.
$\hfill\lhd$
\end{remark}

\subsection{The periodic problem: a Li\'{e}nard type equation}\label{section4.2}

In this section we deal with the existence of periodic positive solutions
to a \textit{Li\'{e}nard type equation}, namely positive solutions of
\begin{equation}\label{BVP-Lienard}
\begin{cases}
\, u'' + h(u) u' + a(x) g(u) = 0, \quad 0 < x < T, \\
\, u(0) = u(T), \quad
\, u'(0) = u'(T),
\end{cases}
\end{equation}
where $h \colon {\mathbb{R}}^{+} \to {\mathbb{R}}$ is a continuous function.
As a preliminary remark, we observe that, if $u(x) > 0$ is any solution of \eqref{BVP-Lienard},
then $\int_{0}^{T} h(u(x)) u'(x) ~\!dx = 0$ and also $\int_{0}^{T} {h(u(x))} u'(x) / {g(u(x))}  ~\!dx = 0$.
Consequently, the condition that $a(x)$ changes sign with negative average, which is necessary for $({\mathscr{P}})$
(when $g'(s) > 0$), is still necessary for \eqref{BVP-Lienard}.

For simplicity, in this section we present only an extension of Corollary~\ref{cor-3.1} to the Li\'{e}nard equation.
In particular, we do not consider the alternative approach of Theorem~\ref{MainTheorem2} for $g(s)$ smooth.
Accordingly, applying the results in Section~\ref{section2}, we prove the following theorem.

\begin{theorem}\label{th-Lienard}
Let $h \colon {\mathbb{R}}^{+} \to {\mathbb{R}}$ be continuous and bounded.
Let $g \colon {\mathbb{R}}^{+}\to{\mathbb{R}}^{+}$ be a continuous function, regularly oscillating at zero
and satisfying $(g_{1})$. Assume
\begin{equation}\label{eq-g2}
\lim_{s\to 0^{+}}\dfrac{g(s)}{s}=0 \quad \text{ and } \quad \lim_{s\to +\infty}\dfrac{g(s)}{s}=+\infty.
\end{equation}
Let $a\colon \mathopen{[}0,T\mathclose{]}\to {\mathbb{R}}$ be a continuous function satisfying $(a_{1})$ and $(a_{2})$.
Then problem \eqref{BVP-Lienard} has at least one positive solution.
\end{theorem}

\begin{proof} We follow the same pattern as the proof of Theorem~\ref{MainTheorem}. In particular,
we are going to show how to achieve the same main steps and formulas in that proof.

First of all, we define $f\colon\mathopen{[}0,T\mathclose{]}\times {\mathbb{R}}^{+} \times {\mathbb{R}} \to {\mathbb{R}}$ as
\begin{equation*}
f(x,s,\xi) = (h(s)-c)\xi + a(x)g(s), \quad \text{with }\; c:= h(0),
\end{equation*}
and we remark that $f$ is a $L^{1}$-Carath\'{e}odory function satisfying $(f_{1})$, $(f_{2})$ and $(f_{3})$.
In this manner, problem \eqref{BVP-Lienard} is of the form
\begin{equation*}
\begin{cases}
\, u'' + c u' + f(x,u,u') = 0, \quad 0 < x < T, \\
\, u(0) = u(T), \quad
\, u'(0) = u'(T),
\end{cases}
\end{equation*}
which is of the same type of \eqref{BVP-sec2} with $u\mapsto -(u'' + c u')$ as differential operator.
The thesis will be reached using Theorem~\ref{th-2.1} with Remark~\ref{rem-2.1}.
In order to avoid unnecessary repetitions, from now on in the proof, all the solutions that we consider
satisfy the $T$-periodic boundary conditions.

\medskip

\noindent
\textit{Verification of $(H_{r})$.}
Observe that \eqref{eq-fs} is satisfied.
We claim that there exists $r_{0}>0$ such that for all $0<r\leq r_{0}$ and
for all $\vartheta\in\mathopen{]}0,1\mathclose{]}$ there are no positive solutions $u(x)$ of
\begin{equation*}
\, u'' + c u' + \vartheta f(x,u,u') = 0
\end{equation*}
such that $\|u\|_{\infty}=r$.
By contradiction, suppose the claim is not true. Then for all $n\in{\mathbb{N}}$
there exist $0<r_{n} <1/n$, $\vartheta_{n}\in\mathopen{]}0,1\mathclose{]}$ and $u_{n}(x)$ positive solution of
\begin{equation}\label{eq-4.7}
u'' + c u' + \vartheta_{n}(h(u)-c)u' + \vartheta_{n} a(x)g(u)=0
\end{equation}
such that $\|u_{n}\|_{\infty} = r_{n}$.

Integrating \eqref{eq-4.7} on $\mathopen{[}0,T\mathclose{]}$ and using the
periodic boundary conditions, we obtain again \eqref{eq-3.4}.
We define
\begin{equation*}
v_{n}(x) := \dfrac{u_{n}(x)}{\|u_{n}\|_{\infty}}
\end{equation*}
and, dividing \eqref{eq-4.7} by $r_{n} = \|u_{n}\|_{\infty}$, we get
\begin{equation}\label{eq-4.8}
v''_{n} + c v'_{n} + \vartheta_{n} (h(u_{n}(x))-c)v_{n}'+ \vartheta_{n}a(x)q(u_{n}(x))v_{n}=0,
\end{equation}
where $h(u_{n}(x))-c \to 0$ and $q(u_{n}(x)):= {g(u_{n}(x))}/{u_{n}(x)} \to 0$, uniformly on $\mathopen{[}0,T\mathclose{]}$ as $n\to \infty$.
Multiplying equation \eqref{eq-4.8} by $v_{n}$ and integrating on $\mathopen{[}0,T\mathclose{]}$, we obtain
\begin{equation*}
\|v'_{n}\|_{L^{2}}^{2} \leq \|a\|_{L^{1}} \sup_{x\in \mathopen{[}0,T\mathclose{]}} |q(u_{n}(x))|\to 0, \quad \text{ as } \; n\to \infty.
\end{equation*}
Using this information on \eqref{eq-4.8}, we see that $\|v''_{n}\|_{L^{1}}\to 0$ as $n\to \infty$. From this fact and observing that
$v_{n}'$ must vanish at some point (by Rolle's theorem), we obtain that
$v_{n}'(x)\to 0$ (as $n\to \infty$) uniformly on $\mathopen{[}0,T\mathclose{]}$ and thus \eqref{eq-3.8} follows.
{}From \eqref{eq-3.4} and \eqref{eq-3.8} we conclude exactly as in the verification of $(H_{r})$ in the proof of Theorem~\ref{MainTheorem}.

\medskip

\noindent
\textit{Verification of $(H_{R})$.}
We choose the same function $v(x)$ as in the proof of Theorem~\ref{MainTheorem} and observe that
the equation in \eqref{eq-2.10} now reads as
\begin{equation}\label{eq-4.9}
u'' + h(u)u' + a(x) g(u) + \alpha v(x) = 0.
\end{equation}
We also fix a constant $C > 0$ such that
\begin{equation}\label{eq-4.10}
|h(s)|\leq C, \quad \forall\, s \geq 0.
\end{equation}
Following the proof of Theorem~\ref{MainTheorem}, we choose an interval $J$ among the intervals $I_{i}$.
We look for a bound $R_{J}>0$ such that any non-negative solution $u(x)$ of \eqref{eq-4.9}, with $\alpha \geq 0$, satisfies $\max_{x\in J} u(x) < R_{J}$.
For notational convenience, we set $J = \mathopen{[}\sigma,\tau\mathclose{]}$ and let $0<\varepsilon<(\tau-\sigma)/2$ be fixed such that
\begin{equation*}
a(x)\not\equiv0 \quad \text{on } J^{\varepsilon},
\end{equation*}
where $J^{\varepsilon}:=\mathopen{[}\sigma_{0},\tau_{0}\mathclose{]}\subseteq \mathopen{[}\sigma,\tau\mathclose{]}$
with $\sigma_{0} - \sigma = \tau - \tau_{0} = \varepsilon$.

Arguing as in \cite{BoFeZa-inp2014}, we can prove that
$u'(x) \leq u(x) e^{CT}/\varepsilon$, for all $x\in \mathopen{[}\sigma_{0},\tau\mathclose{]}$ such that $u'(x) \geq 0$, and also
$|u'(x)| \leq u(x) e^{CT}/\varepsilon$, for all $x\in \mathopen{[}\sigma,\tau_{0}\mathclose{]}$ such that $u'(x) \leq 0$.
The proof follows the same argument as in \cite{BoFeZa-inp2014}, observing that the auxiliary function
\begin{equation*}
\Phi \colon x \mapsto u'(x) \exp\biggl{(}\int_{0}^{x} h(u(\xi))~\!d\xi\biggr{)}
\end{equation*}
is non-increasing on $J$.

Let $\hat{\lambda}$ be the first (positive) eigenvalue of the eigenvalue problem
\begin{equation*}
\begin{cases}
\, \bigl{(}e^{Cx} \varphi'\bigr{)}' + e^{-Cx} \lambda a(x) \varphi = 0 \\
\, \varphi(\sigma_{0})=\varphi(\tau_{0})=0.
\end{cases}
\end{equation*}
We fix a constant $M>0$ such that
\begin{equation*}
M > \hat{\lambda}.
\end{equation*}
{}From \eqref{eq-g2} it follows that there exists a constant $\tilde{R}=\tilde{R}(M)>0$ such that
\begin{equation*}
g(s) > M s, \quad \forall \, s\geq \tilde{R}.
\end{equation*}

By contradiction, suppose there is not a constant $R_{J}>0$ with the properties listed above. So, for each integer $n>0$
there exists a solution $u_{n}\geq0$ of \eqref{eq-4.9} with $\max_{x\in J}u_{n}(x)=:\hat{R}_{n}>n$.
For each $n > \tilde{R}$ we take $\hat{x}_{n} \in J$ such that $u_{n}(\hat{x}_{n}) = \hat{R}_{n}$ and let
$\mathopen{]}\varsigma_{n},\omega_{n}\mathclose{[}\subseteq J$ be the intersection with $\mathopen{]}\sigma,\tau\mathclose{[}$
of the maximal open interval containing $\hat{x}_{n}$ and such that $u_{n}(x) > \tilde{R}$ for all $x\in\mathopen{]}\varsigma_{n},\omega_{n}\mathclose{[}$.
We fix an integer $N$ such that
\begin{equation*}
N > \tilde{R} + \dfrac{\tilde{R} \, T e^{2CT}}{\varepsilon}
\end{equation*}
and we claim that $\mathopen{]}\varsigma_{n},\omega_{n}\mathclose{[} \supseteq \mathopen{[}\sigma_{0},\tau_{0}\mathclose{]}$, for each $n \geq N$.
Suppose by contradiction that $\sigma_{0} \leq \varsigma_{n}$. In this case, we find that $u_{n}(\varsigma_{n}) =\tilde{R}$
and $u'_{n}(\varsigma_{n}) \geq 0$.
Moreover, $u'_{n}(\varsigma_{n}) \leq \tilde{R}e^{CT}/\varepsilon$. Using the monotonicity
of $\Phi$, $\Phi(x) \leq \Phi(\varsigma_{n})$ for every $x \in \mathopen{[}\varsigma_{n},\hat{x}_{n}\mathclose{]}$
and therefore, using also \eqref{eq-4.10}, we find $u'(x) \leq \tilde{R}e^{2CT}/\varepsilon$
for every $x \in \mathopen{[}\varsigma_{n},\hat{x}_{n}\mathclose{]}$.
Finally, an integration on $\mathopen{[}\varsigma_{n},\hat{x}_{n}\mathclose{]}$ yields
\begin{equation*}
n < \hat{R}_{n} = u_{n}(\hat{x}_{n}) \leq \tilde{R} + \dfrac{\tilde{R} \, T e^{2CT}}{\varepsilon},
\end{equation*}
hence a contradiction, since $n\geq N$. A symmetric argument provides a contradiction if we suppose that $\omega_{n} \leq \tau_{0}$.
This proves the claim.

So, we can fix an integer $N > \tilde{R}$
such that $u_{n}(x)>\tilde{R}$ for every $x\in J^{\varepsilon}$ and for $n\geq N$. The function $u_{n}(x)$, being a solution of
equation \eqref{eq-4.9}, also satisfies
\begin{equation*}
\begin{cases}
\, u_{n}'(x) = \dfrac{y_{n}(x)}{p_{n}(x)} \vspace*{2pt} \\
\, y_{n}'(x) = - H_{n}(x,u_{n}(x)),
\end{cases}
\end{equation*}
where
\begin{equation*}
p_{n}(x):= \exp\biggl{(} \int_{0}^{x} h(u_{n}(\xi))~\!d\xi \biggr{)}
\end{equation*}
and
\begin{equation*}
H_{n}(x,u_{n}(x)):= \exp\biggl{(} \int_{0}^{x} h(u_{n}(\xi))~\!d\xi \biggr{)}\Bigl(a(x) g(u_{n}(x)) + \alpha v(x)\Bigr).
\end{equation*}
Passing to the polar coordinates, via a Pr\"{u}fer transformation, we consider
\begin{equation*}
p_{n}(x)u'_{n}(x) = r_{n}(x) \cos \vartheta_{n}(x), \qquad u_{n}(x) = r_{n}(x) \sin \vartheta_{n}(x),
\end{equation*}
and obtain, for every $x\in J^{\varepsilon}$, that
\begin{equation*}
\begin{aligned}
\vartheta'_{n}(x) &= \dfrac{\cos^{2} \vartheta_{n}(x)}{p_{n}(x)} + \dfrac{H_{n}(x,u_{n}(x))}{u_{n}(x)} \sin^{2}\vartheta_{n}(x)
\\                &\geq \dfrac{\cos^{2} \vartheta_{n}(x)}{p_{n}(x)} + M p_{n}(x) a(x) \sin^{2}\vartheta_{n}(x).
\end{aligned}
\end{equation*}
We also consider the linear equation
\begin{equation}\label{eq-4.13}
\bigl{(}e^{Cx} u'\bigr{)}' + e^{-Cx} M a(x) u = 0
\end{equation}
and its associated angular coordinate $\vartheta(x)$ (via the Pr\"{u}fer transformation), which satisfies
\begin{equation*}
\vartheta'(x) = \dfrac{\cos^{2} \vartheta(x)}{e^{Cx}} + e^{-Cx} M a(x) \sin^{2}\vartheta(x).
\end{equation*}
Note also that the angular functions $\vartheta_{n}$ and $\vartheta$ are non-decreasing in $J^{\varepsilon}$.
Using a classical comparison result in the frame of Sturm's theory (cf.~\cite[Chap.~8, Theorem~1.2]{CoLe-1955}), we
find that
\begin{equation}\label{eq-4.14}
\vartheta_{n}(x) \geq  \vartheta(x),  \quad \forall \, x\in J^{\varepsilon},
\end{equation}
if we choose $\vartheta(\sigma_{0}) = \vartheta_{n}(\sigma_{0})$.
Consider now a fixed $n \geq N$. Since $u_{n}(x) \geq \tilde{R}$ for every $x \in J^{\varepsilon}$, we must have
\begin{equation}\label{eq-4.15}
\vartheta_{n}(x)\in\mathopen{]}0,\pi\mathclose{[},\quad \forall \, x\in J^{\varepsilon}.
\end{equation}
On the other hand, by the choice of $M>0$, we know that any non-negative solution $u(x)$ of \eqref{eq-4.13} with $u(\sigma_{0}) > 0$
must vanish at some point
in $\mathopen{]}\sigma_{0},\tau_{0}\mathclose{[}$ (see \cite[Chap.~8, Theorem~1.1]{CoLe-1955}).
Therefore, from $\vartheta(\sigma_{0}) = \vartheta_{n}(\sigma_{0}) \in \mathopen{]}0,\pi\mathclose{[}$,
we conclude that there exists $x^{*}\in \mathopen{]}\sigma_{0},\tau_{0}\mathclose{[}$ such that $\vartheta(x^{*}) = \pi$.
By \eqref{eq-4.14} we have that $\vartheta_{n}(x^{*}) \geq \pi$, which contradicts \eqref{eq-4.15}.

By the arbitrary choice of $J$ among the intervals $I_{1},\ldots,I_{m}$, for each $i=1,\ldots,m$ we obtain the existence of a constant $R_{I_{i}}>0$
such that any non-negative solution $u(x)$ of \eqref{eq-4.9}, with $\alpha \geq 0$, satisfies $\max_{x\in I_{i}} u(x) < R_{I_{i}}$.
Finally, let us fix a constant $R > r_{0}$ (with $r_{0}$ coming from the first part of the proof) as in \eqref{eq-R}, so that
$R \geq R_{I_{i}}$ for all $i=1,\ldots,m$.

Consider now a (maximal) interval $\mathcal{J}$ contained in $\mathopen{[}0,T\mathclose{]} \setminus \bigcup_{i=1}^{m} I_{i}$ where
$a(x) \leq 0$. For simplicity in the exposition, we suppose that $\mathcal{J}$ lies between two intervals $I_{i}$ where $a(x)\geq 0$,
so that $\mathcal{J}= \mathopen{]}\tau',\sigma'\mathclose{[}$, with $\tau' \in I_{k}$ and $\sigma'\in I_{k+1}$.

Let $u(x)$ be a non-negative solution of \eqref{eq-4.9}.
For $x\in \mathcal{J}$, equation \eqref{eq-4.9} reads as
\begin{equation*}
u'' + h(u)u' + a(x) g(u) = 0
\end{equation*}
and therefore the auxiliary function $\Phi$ is non-decreasing on $\mathcal{J}$. If $u'(x^{*}) \geq 0$,
for some $x^{*} \in \mathopen{[}\tau',\sigma'\mathclose{[}$, then $u'(x) \geq 0$ for all $x\in \mathopen{[}x^{*},\sigma'\mathclose{]}$,
hence $u(x^{*}) \leq u(\sigma') < R$ (because $\sigma'$ belongs to some interval $I_{i}$, where $u(x)$ is bounded by $R$).
Similarly, if $u'(x^{*}) \leq 0$, for some $x^{*} \in \mathopen{]}\tau',\sigma'\mathclose{]}$,
then $u'(x) \leq 0$ for all $x\in \mathopen{[}\tau',x^{*}\mathclose{]}$,
hence $u(x^{*}) \leq u(\tau') < R$ (because $\tau'$ belongs to some interval $I_{i}$).
Thus, we easily deduce that $u(x) < R$ for all $x\in \text{\rm cl}(\mathcal{J}) = \mathopen{[}\tau',\sigma'\mathclose{]}$.
The same argument can be easily adapted if $\mathcal{J} = \mathopen{[}0,\sigma'\mathclose{[}$ or
$\mathcal{J} = \mathopen{]}\tau',T\mathclose{]}$ (with, respectively, $\sigma' \in I_{1}$ or $\tau' \in I_{m}$),
using the $T$-periodic boundary conditions.

In this manner, we have found a constant $R > r_{0}$ such that any non-negative solution $u(x)$ of
\eqref{eq-4.9}, with $\alpha \geq 0$, satisfies
\begin{equation*}
\|u\|_{\infty} < R.
\end{equation*}
This shows that the first part of $(H_{R})$ is valid independently of the choice of $\alpha_{0}$.

Now we fix $\alpha_{0}$ as in \eqref{eq-3.11}. It remains to verify that for
$\alpha = \alpha_{0}$ there are no solutions $u(x)$ of \eqref{eq-4.9} with
$0\leq u(x) \leq R$ on $\mathopen{[}0,T\mathclose{]}$. Indeed, if there were, integrating on $\mathopen{[}0,T\mathclose{]}$ the differential equation
and using the boundary conditions, we obtain
\begin{equation*}
\alpha \|v\|_{L^{1}} = \alpha \int_{0}^{T}v(x)~\!dx \leq \int_{0}^{T}|a(x)|g(u(x))~\!dx \leq \|a\|_{L^{1}}\,\max_{0\leq s \leq R}g(s),
\end{equation*}
which leads to a contradiction with respect to the choice of $\alpha_{0}$. Thus $(H_{R})$ is verified.

\medskip

\noindent
Having verified $(H_{r})$ and $(H_{R})$, the thesis follows from Theorem~\ref{th-2.1} with Remark~\ref{rem-2.1}.
\end{proof}

\section{Final remarks}\label{section5}

In the setting of the Dirichlet (two-point) boundary value problem associated to equation
\begin{equation*}
u'' + a(x) g(u) = 0,
\end{equation*}
it is known that, in the superlinear case, multiple positive solutions can be obtained when the
weight is sufficiently negative in some intervals.
More precisely, writing explicitly the dependence of $a(x)$ on a real parameter $\mu > 0$ which controls the
negative part, the following result can be given for the boundary value problem
\begin{equation}\label{eq-5.1}
\begin{cases}
\, u'' + a_{\mu}(x) g(u) = 0 \\
\, u(0) = u(T) = 0,
\end{cases}
\end{equation}
with
\begin{equation*}
a_{\mu}(x):= a^{+}(x) - \mu \, a^{-}(x).
\end{equation*}

\begin{theorem}\label{th-5.1}
Let $g(s)$ be a continuous function satisfying $(g_{1})$ and such that
\begin{equation*}
\lim_{s\to 0^{+}} \dfrac{g(s)}{s} = 0, \qquad \lim_{s\to +\infty} \dfrac{g(s)}{s} = +\infty.
\end{equation*}
Then there exists $\mu^{*}> 0$ such that for each $\mu > \mu^{*}$ problem \eqref{eq-5.1} has at least
$2^{m} -1$ positive solutions, where $m \geq 1$ is the number of positive humps of the weight function which are
separated by $m-1$ negative humps.
\end{theorem}

For this result and the precise technical assumptions which are needed, see \cite{FeZa-pre2014}.
In this setting (due to the boundary conditions), by positive solutions we mean solutions which are positive on $\mathopen{]}0,T\mathclose{[}$.
Previous versions of this theorem have been obtained for $g(s) = s^{\gamma}$ with $\gamma > 1$
in \cite{GaHaZa-2003mod, GaHaZa-2004} in the ODE case (using the shooting method)
and in \cite{BoGoHa-2005} for PDEs (using a variational approach). Further progresses in this direction
have been achieved in \cite{GiGo-2009rse, GiGo-2009jde} for the Dirichlet problem for PDEs.

Concerning the boundary conditions considered in the present paper, multiplicity results have been recently provided
in \cite{Bo-2011} for the Neumann problem (using the shooting method) and in \cite{BaBoVe-pre2014}
for the periodic problem (using a variational approach). It seems reasonable to adapt our arguments to problem
\begin{equation*}
\begin{cases}
\, u'' + a_{\mu}(x) g(u) = 0, \quad 0 < x < T, \\
\, {\mathscr{B}}(u,u') = \underline{0},
\end{cases}
\end{equation*}
in order to obtain multiplicity results when $\mu > 0$ is sufficiently large.

\section{Appendix}\label{Appendix}

In this section we present two results concerning the solutions of the boundary value problem
\begin{equation}\label{BVP-Appendix}
\begin{cases}
\, u''+h(x,u)=0 \\
\, {\mathscr{B}}(u,u') = \underline{0},
\end{cases}
\end{equation}
where $h\colon \mathopen{[}0,T\mathclose{]}\times {\mathbb{R}}\to {\mathbb{R}}$ is a $L^{1}$-Ca\-ra\-th\'{e}o\-dory function.
As in rest of the paper, by ${\mathscr{B}}(u,u') = \underline{0}$ we mean the Neumann or the periodic boundary conditions
on $\mathopen{[}0,T\mathclose{]}$.

The first result is a \textit{maximum principle} that ensures the non-negativity or the positivity of the solutions to problem \eqref{BVP-Appendix}.
In the applications, for example, we have $h(x,s)= a(x)g(s)$.

\begin{lemma}\label{Maximum principle}
Let $h\colon \mathopen{[}0,T\mathclose{]}\times {\mathbb{R}}\to {\mathbb{R}}$ be a $L^{1}$-Ca\-ra\-th\'{e}o\-dory function.
\begin{itemize}
\item [$(i)$]
    If
    \begin{equation*}
    h(x,s)>0, \quad \text{a.e. } x\in\mathopen{[}0,T\mathclose{]}, \text{ for all } s<0,
    \end{equation*}
    then any solution of \eqref{BVP-Appendix} is non-negative on $\mathopen{[}0,T\mathclose{]}$.
\item [$(ii)$]
    If $h(x,0)\equiv0$ and
    there exists $q\in L^{1}(\mathopen{[}0,T\mathclose{]},{\mathbb{R}}^{+})$ such that
    \begin{equation*}
     \limsup_{s\to0^{+}}\dfrac{|h(x,s)|}{s} \leq q(x), \quad \text{uniformly a.e. } x\in\mathopen{[}0,T\mathclose{]},
    \end{equation*}
    then every nontrivial non-negative solution $u(x)$ of \eqref{BVP-Appendix} satisfies $u(x)>0$,
    for all $x\in\mathopen{[}0,T\mathclose{]}$.
\end{itemize}
\end{lemma}

\begin{proof}

\noindent $(i)$.\;
By contradiction, suppose that there exists a solution $u(x)$ of \eqref{BVP-Appendix}
and $\hat{x}\in\mathopen{[}0,T\mathclose{]}$ such that $u(\hat{x})<0$.
Let $\mathopen{]}x_{1},x_{2}\mathclose{[}\subseteq\mathopen{]}0,T\mathclose{[}$ be the maximal open interval containing $\hat{x}$
with $u(x)<0$, for all $x_{1}<x<x_{2}$.
Since $u''(x) < 0$ for a.e.~$x\in\mathopen{[}x_{1},x_{2}\mathclose{]}$,
an elementary convexity argument shows that $0 < x_{1} < x_{2} < T$ is not possible.
Similarly, also $u(x) < 0$ for all $x\in \mathopen{]}0,T\mathclose{[}$ can be excluded, otherwise,
$0 > \int_{0}^{T} u''(x) ~\!dx = u'(T) - u'(0)$, contradicting the boundary conditions.
Hence, there are only two possibilities: either $x_{1}=0$ and $x_{2} < T$, or $0 < x_{1}$ and $x_{2}=T$.
Suppose $x_{1}=0$ (the other case can be treated in a similar manner). In this case, $u(0) \leq 0$
and moreover $u'(0) > 0$ (otherwise, by concavity, one has $u(x) < 0$ for all $x\in \mathopen{]}0,T\mathclose{]}$,
a situation previously excluded). This already gives a contradiction with the Neumann boundary condition at $x=0$.
On the other hand, if we consider the periodic boundary condition, we have that $u(T)=u(0) \leq 0$ and $u'(T)=u'(0) > 0$.
Hence, by the concavity of $u$ on the intervals where $u(x)<0$, we obtain that $u(x)<0$ for every $x\in\mathopen{[}0,T\mathclose{[}$, a contradiction.

\smallskip

\noindent $(ii)$.\;
By contradiction, suppose that there exists a solution $u(x)\geq 0$ of \eqref{BVP-Appendix}
and $x^{*}\in\mathopen{[}0,T\mathclose{]}$ such that $u(x^{*})=0$ (so, $u'(x^{*})=0$).

We claim that there exists $\varepsilon>0$ such that $u(x)=0$, for all $x\in\mathopen{[}x^{*}-\varepsilon,x^{*}+\varepsilon\mathclose{]}$.
So that $u\equiv0$ on $\mathopen{[}0,T\mathclose{]}$, a contradiction.

{} From the hypotheses, we obtain that there exists $\delta>0$ such that
\begin{equation*}
|h(x,s)|\leq q_{1}(x)s, \quad \text{a.e. } x\in\mathopen{[}0,T\mathclose{]}, \; \forall \, 0 \leq s \leq \delta,
\end{equation*}
where $q_{1}(x):= q(x)+1$.
Using the continuity of $u(x)$, we fix $\varepsilon>0$ such that
$0\leq u(x)\leq\delta$, for all $x\in\mathopen{[}x^{*}-\varepsilon,x^{*}+\varepsilon\mathclose{]}$.

We use $\|(\xi_{1},\xi_{2})\| = |\xi_{1}| + |\xi_{2}|$ as a standard norm in ${\mathbb{R}}^{2}$.
For all $x\in\mathopen{]}x^{*},x^{*}+\varepsilon\mathclose{]}$ we have
\begin{equation*}
\begin{aligned}
0 \leq \|(u(x),u'(x))\| & = |u(x)| + |u'(x)| = u(x) + |u'(x)| =
\\ &   =  u(x^{*})+\int_{x^{*}}^{x}u'(\xi)~\!d\xi+\biggl{|}u'(x^{*})+\int_{x^{*}}^{x}-h(\xi,u(\xi))~\!d\xi\biggr{|}
\\ & \leq  \int_{x^{*}}^{x}|u'(\xi)|~\!d\xi+\int_{x^{*}}^{x}|h(\xi,u(\xi))|~\!d\xi
\\ & \leq  \int_{x^{*}}^{x}\Bigl[q_{1}(\xi)|u(\xi)|+|u'(\xi)|\Bigr]~\!d\xi
\\ & \leq  \int_{x^{*}}^{x}(q_{1}(\xi)+1)(|u(\xi)|+|u'(\xi)|)~\!d\xi.
\end{aligned}
\end{equation*}
Using the classical Gronwall's inequality, we attain
\begin{equation*}
0\leq u(x)\leq \|(u(x),u'(x))\| =0, \quad \forall \, x\in\mathopen{]}x^{*},x^{*}+\varepsilon\mathclose{]}.
\end{equation*}
With an analogous computation one can prove that $u(x)=0$ for all $x\in\mathopen{[}x^{*}-\varepsilon,x^{*}{[}$.
Hence the claim and $(ii)$ are proved.
\end{proof}

\begin{remark}\label{rem-6.1}
The maximum principle just presented can be also stated for the more general boundary value problem
\begin{equation*}
\begin{cases}
\, u'' + \tilde{f}(x,u,u') = 0, \quad 0 < x < T, \\
\, {\mathscr{B}}(u,u') = \underline{0},
\end{cases}
\end{equation*}
where $\tilde{f}\colon\mathopen{[}0,T\mathclose{]}\times {\mathbb{R}} \times {\mathbb{R}} \to {\mathbb{R}}$ is a $L^{p}$-Carath\'{e}odory function
as in Section~\ref{section2.2}, hence equal to $-s$ for $s\leq 0$ and satisfying conditions $(f_{1})$ and $(f_{2})$.
The proof of this result is the same as that just viewed with minor changes.
$\hfill\lhd$
\end{remark}

The following result provides a priori bounds for non-negative solutions on the intervals where $h(x,s)$ is non-negative.
This lemma is used in the verification of condition $(H_{R})$ in Theorem~\ref{MainTheorem} and Theorem~\ref{MainTheorem2}.
In a way, it is employed to compute the coincidence degree on large balls.

\begin{lemma}\label{lemma_RJ}
Let $h\colon \mathopen{[}0,T\mathclose{]}\times {\mathbb{R}}\to {\mathbb{R}}$ be a $L^{1}$-Ca\-ra\-th\'{e}o\-dory function.
Suppose there exists a closed interval $J \subseteq \mathopen{[}0,T\mathclose{]}$ such that
\begin{equation*}
h(x,s)\geq 0, \quad \text{a.e. } x\in J, \; \forall \, s\geq0;
\end{equation*}
and there is a measurable function $q_{\infty}\in L^{1}(J,{\mathbb{R}}^{+})$ with $q_{\infty}\not\equiv0$, such that
\begin{equation}\label{condinfty}
\liminf_{s\to+\infty}\dfrac{h(x,s)}{s}\geq q_{\infty}(x), \quad \text{uniformly a.e. } x\in J.
\end{equation}
Let $\mu_{J}$ be the first positive eigenvalue of the eigenvalue problem
\begin{equation*}
\varphi'' + \lambda q_{\infty}(x) \varphi = 0, \quad \varphi|_{\partial J}=0,
\end{equation*}
and suppose that $\mu_{J}<1$.
Then there exists $R_{J}>0$ such that for each Ca\-ra\-th\'{e}o\-dory function
$k\colon \mathopen{[}0,T\mathclose{]} \times{\mathbb{R}}^{+}\to {\mathbb{R}}$ with
\begin{equation*}
k(x,s)\geq h(x,s), \quad \text{a.e. } x\in J, \; \forall \, s\geq0,
\end{equation*}
every solution $u(x)\geq0$ of the BVP
\begin{equation}\label{BVP-k}
\begin{cases}
\, u''+ k(x,u)=0 \\
\, {\mathscr{B}}(u,u') = \underline{0}
\end{cases}
\end{equation}
satisfies $\max_{x\in J}u(x)<R_{J}$.
\end{lemma}

We stress that the constant $R_{J}$ does not depend on the function $k(x,s)$.
Notice also that our assumptions are ``local'', in the sense that we do not require their validity on the whole domain.

\begin{proof}
Just to fix a notation along the proof, we set $J:=\mathopen{[}x_{1},x_{2}\mathclose{]}$.
By contradiction, suppose there is not a constant $R_{J}$ with those properties. So, for all $n>0$
there exists $\tilde{u}_{n}\geq0$ solution of \eqref{BVP-k} with $\max_{x\in J}\tilde{u}_{n}(x)=:\hat{R}_{n}>n$.

Let $q_{n}(x)$ be a monotone non-decreasing sequence of non-negative measurable functions such that
\begin{equation*}
h(x,s) \geq q_{n}(x) s, \quad \text{a.e. } x\in J, \; \forall \, s \geq n,
\end{equation*}
and $q_{n}\to q_{\infty}$ uniformly almost everywhere in $J$.
The existence of such a sequence comes from condition \eqref{condinfty}.

Fix $\varepsilon<(1-\mu_{J})/2$. Hence, there exists an integer $N >0$ such $q_{n}\not\equiv 0$ for each $n\geq N$ and
\begin{equation*}
\nu_{n} \leq 1-\varepsilon , \quad \forall \, n\geq N,
\end{equation*}
where $\nu_{n}>0$ is the first positive eigenvalue of the eigenvalue problem
\begin{equation*}
\varphi'' + \lambda q_{n}(x) \varphi = 0, \quad \varphi|_{\partial J}=0.
\end{equation*}
Now we fix $N$ as above and denote by
$\varphi$ the positive eigenfunction of
\begin{equation*}
\begin{cases}
\, \varphi''+ \nu_{N} q_{N}(x)\varphi=0 \\
\, \varphi(x_{1})=\varphi(x_{2})=0 ,
\end{cases}
\end{equation*}
with $\|\varphi\|_{\infty}=1$. Then $\varphi(x)>0$, $\forall \, x\in\mathopen{]}x_{1},x_{2}\mathclose{[}$,
and $\varphi'(x_{1})>0>\varphi'(x_{2})$.

For each $n\geq N$, let $J'_{n}\subseteq J$ be the maximal closed interval, such that
\begin{equation*}
\tilde{u}_{n}(x)\geq N, \quad \forall \, x\in J'_{n}.
\end{equation*}
By the concavity of the solution in the interval $J$ and the definition of $J'_{n}$,
we also have that
\begin{equation*}
 \tilde{u}_{n}(x)\leq N, \quad \forall \, x\in J\setminus J'_{n}.
\end{equation*}
Another consequence of the concavity of $\tilde{u}_{n}$ on $J$ ensures that
\begin{equation*}
\tilde{u}_{n}(x)\geq \dfrac{\hat{R}_{n}}{x_{2}-x_{1}}\min\{x-x_{1},x_{2}-x\}, \quad \forall x\in J,
\end{equation*}
(see~\cite{GaHaZa-2003} for a similar estimate).
Hence, if we take $n\geq 2N$, we find that $\tilde{u}_{n}(x)\geq N$, for all $x$ in the well-defined closed interval
\begin{equation*}
A_{n}:=\biggl{[}x_{1}+\dfrac{N}{\hat{R}_{n}}(x_{2}-x_{1}),x_{2}-\dfrac{N}{\hat{R}_{n}}(x_{2}-x_{1})\biggr{]}\subseteq J'_{n}.
\end{equation*}
By construction, ${\it meas}(J\setminus J'_{n})\leq {\it meas}(J\setminus A_{n})\to 0$ as $n\to \infty$.

Using a Sturm comparison argument, for each $n\geq N$, we obtain
\begin{equation*}
\begin{aligned}
0  & \geq \tilde{u}_{n}(x_{2})\varphi'(x_{2})-\tilde{u}_{n}(x_{1})\varphi'(x_{1})
 = \Bigl{[}\tilde{u}_{n}(x)\varphi'(x)-\tilde{u}'_{n}(x)\varphi(x)\Bigr{]}_{x=x_{1}}^{x=x_{2}}
\\ & = \int_{x_{1}}^{x_{2}}\dfrac{d}{dx}\Bigl{[}\tilde{u}_{n}(x)\varphi'(x)-\tilde{u}'_{n}(x)\varphi(x)\Bigr{]} ~\!dx
\\ & = \int_{J}\Bigl{[}\tilde{u}_{n}(x)\varphi''(x)-\tilde{u}''_{n}(x)\varphi(x)\Bigr{]} ~\!dx
\\ & = \int_{J}\Bigl{[}-\tilde{u}_{n}(x)\nu_{N}q_{N}(x)\varphi(x)+k(x,\tilde{u}_{n}(x))\varphi(x)\Bigr{]} ~\!dx
\\ & =  \int_{J}\Bigl{[}k(x,\tilde{u}_{n}(x))-\nu_{N}q_{N}(x)\tilde{u}_{n}(x)\Bigr{]}\varphi(x) ~\!dx
\\ & \geq \int_{J}\Bigl{[}h(x,\tilde{u}_{n}(x))-\nu_{N}q_{N}(x)\tilde{u}_{n}(x)\Bigr{]}\varphi(x) ~\!dx
\\ & = \int_{J'_{n}}\Bigl{[}h(x,\tilde{u}_{n}(x))-q_{N}(x)\tilde{u}_{n}(x)\Bigr{]}\varphi(x) ~\!dx +
(1-\nu_{N})\int_{J'_{n}}q_{N}(x)\tilde{u}_{n}(x)\varphi(x) ~\!dx
\\ & \quad + \int_{J\setminus J'_{n}}\Bigl{[}h(x,\tilde{u}_{n}(x))-\nu_{N}q_{N}(x)\tilde{u}_{n}(x)\Bigr{]}\varphi(x) ~\!dx.
\end{aligned}
\end{equation*}
Recalling that
\begin{equation*}
h(x,s)\geq q_{N}(x) s, \quad \text{a.e. } x\in J, \; \forall \, s\geq N,
\end{equation*}
we know that
\begin{equation*}
h(x,\tilde{u}_{n}(x))-q_{N}(x)\tilde{u}_{n}(x)\geq 0, \quad \text{a.e. } x\in J'_{n}, \; \forall \, n\geq N.
\end{equation*}
Then, using the Carath\'{e}odory assumption, which implies that
\begin{equation*}
|h(x,s)|\leq \gamma_{N}(x), \quad \text{a.e. } x\in J, \; \forall \, 0\leq s \leq N,
\end{equation*}
where $\gamma_{N}$ is a suitably non-negative integrable function, we obtain
\begin{equation*}
\begin{aligned}
0  & \geq
\int_{J'_{n}}\Bigl{[}h(x,\tilde{u}_{n}(x))-q_{N}(x)\tilde{u}_{n}(x)\Bigr{]}\varphi(x) ~\!dx +
(1-\nu_{N})\int_{J'_{n}}q_{N}(x)\tilde{u}_{n}(x)\varphi(x) ~\!dx
\\ & \quad + \int_{J\setminus J'_{n}}\Bigl{[}h(x,\tilde{u}_{n}(x))-\nu_{N}q_{N}(x)\tilde{u}_{n}(x)\Bigr{]}\varphi(x) ~\!dx
\\ & \geq
\varepsilon N \int_{J'_{n}}q_{N}(x)\varphi(x) ~\!dx + \int_{J\setminus J'_{n}}\Bigl{[}-\gamma_{N}(x)-N\nu_{N}q_{N}(x)\Bigr{]} ~\!dx
\\ & =
\varepsilon N \int_{J}q_{N}(x)\varphi(x) ~\!dx -  \varepsilon N\int_{J\setminus J'_{n}} q_{N}(x)\varphi(x) ~\!dx
\\ & \quad - \int_{J\setminus J'_{n}}\Bigl{[}\gamma_{N}(x) + N\nu_{N}q_{N}(x)\Bigr{]} ~\!dx.
\end{aligned}
\end{equation*}
Passing to the limit as $n\to \infty$ and using the dominated convergence theorem, we obtain
\begin{equation*}
0 \geq \varepsilon N \int_{J}q_{N}(x)\varphi(x) ~\!dx > 0,
\end{equation*}
a contradiction.
\end{proof}

\begin{remark}\label{rem-6.3}
We note that the Neumann or the periodic boundary condition in problem \eqref{BVP-k} has no role in the proof of Lemma~\ref{lemma_RJ}.
In fact, the key point is that we deal only with non-negative solutions of the equation $u''+k(x,u)=0$.
Consequently, the same thesis holds also when the relation ${\mathscr{B}}(u,u') = \underline{0}$ defines any boundary condition.
$\hfill\lhd$
\end{remark}

\bibliographystyle{elsart-num-sort}
\bibliography{Feltrin_Zanolin_biblio}

\begin{thebibliography}{10}
\expandafter\ifx\csname url\endcsname\relax
  \def\url#1{\texttt{#1}}\fi
\expandafter\ifx\csname urlprefix\endcsname\relax\def\urlprefix{URL }\fi

\bibitem{Ac-2009}
N.~Ackermann, Long-time dynamics in semilinear parabolic problems with
  autocatalysis, in: Recent progress on reaction-diffusion systems and
  viscosity solutions, World Sci. Publ., Hackensack, NJ, 2009, pp. 1--30.

\bibitem{BaPoTe-1987}
C.~Bandle, M.~A. Pozio, A.~Tesei, The asymptotic behavior of the solutions of
  degenerate parabolic equations, Trans. Amer. Math. Soc. 303 (1987) 487--501.

\bibitem{BaPoTe-1988}
C.~Bandle, M.~A. Pozio, A.~Tesei, Existence and uniqueness of solutions of
  nonlinear {N}eumann problems, Math. Z. 199 (1988) 257--278.

\bibitem{BaBoVe-pre2014}
V.~L. Barutello, A.~Boscaggin, G.~Verzini, Positive solutions with a complex
  behavior for superlinear indefinite {ODE}s on the real line, J. Differential
  Equations 259 (2015) 3448--3489.

\bibitem{BeCaDoNi-1994}
H.~Berestycki, I.~Capuzzo-Dolcetta, L.~Nirenberg, Superlinear indefinite
  elliptic problems and nonlinear {L}iouville theorems, Topol. Methods
  Nonlinear Anal. 4 (1994) 59--78.

\bibitem{BeCaDoNi-1995}
H.~Berestycki, I.~Capuzzo-Dolcetta, L.~Nirenberg, Variational methods for
  indefinite superlinear homogeneous elliptic problems, NoDEA Nonlinear
  Differential Equations Appl. 2 (1995) 553--572.

\bibitem{Be-1983}
S.~M. Berman, High level sojourns of a diffusion process on a long interval, Z.
  Wahrsch. Verw. Gebiete 62 (1983) 185--199.

\bibitem{BiGoTe-1987}
N.~H. Bingham, C.~M. Goldie, J.~L. Teugels, Regular variation, vol.~27 of
  Encyclopedia of Mathematics and its Applications, Cambridge University Press,
  Cambridge, 1987.

\bibitem{BoGoHa-2005}
D.~Bonheure, J.~M. Gomes, P.~Habets, Multiple positive solutions of superlinear
  elliptic problems with sign-changing weight, J. Differential Equations 214
  (2005) 36--64.

\bibitem{Bo-2011}
A.~Boscaggin, A note on a superlinear indefinite {N}eumann problem with
  multiple positive solutions, J. Math. Anal. Appl. 377 (2011) 259--268.

\bibitem{Bo-2012}
A.~Boscaggin, One-signed harmonic solutions and sign-changing subharmonic
  solutions to scalar second order differential equations, Adv. Nonlinear Stud.
  12 (2012) 445--463.

\bibitem{BoFeZa-inp2014}
A.~Boscaggin, G.~Feltrin, F.~Zanolin, Pairs of positive periodic solutions of
  nonlinear {ODEs} with indefinite weight: a topological degree approach for
  the super-sublinear case, Proc. Roy. Soc. Edinburgh Sect. A{}, to appear.

\bibitem{BoGa-pre2014}
A.~Boscaggin, M.~Garrione, Multiple solutions to {N}eumann problems with
  indefinite weight and bounded nonlinearities, J. Dynam. Differential
  Equations{}, to appear.

\bibitem{BoZa-2012}
A.~Boscaggin, F.~Zanolin, Pairs of positive periodic solutions of second order
  nonlinear equations with indefinite weight, J. Differential Equations 252
  (2012) 2900--2921.

\bibitem{BoZa-2013}
A.~Boscaggin, F.~Zanolin, Second-order ordinary differential equations with
  indefinite weight: the {N}eumann boundary value problem, Ann. Mat. Pura Appl.
  (4) 194 (2015) 451--478.

\bibitem{CoLe-1955}
E.~A. Coddington, N.~Levinson, Theory of ordinary differential equations,
  McGraw-Hill Book Company, Inc., New York-Toronto-London, 1955.

\bibitem{DCHa-2006}
C.~De~Coster, P.~Habets, Two-point boundary value problems: lower and upper
  solutions, vol. 205 of Mathematics in Science and Engineering, Elsevier B.
  V., Amsterdam, 2006.

\bibitem{DjTo-2001}
D.~Djur{\v{c}}i{\'c}, A.~Torga{\v{s}}ev, Strong asymptotic equivalence and
  inversion of functions in the class {$K_c$}, J. Math. Anal. Appl. 255 (2001)
  383--390.

\bibitem{FeZa-pre2014}
G.~Feltrin, F.~Zanolin, Multiple positive solutions for a superlinear problem:
  a topological approach, J. Differential Equations 259 (2015) 925--963.

\bibitem{Fo-1984}
G.~B. Folland, Real analysis, Pure and Applied Mathematics (New York), John
  Wiley \& Sons, Inc., New York, 1984.

\bibitem{GaMa-1977}
R.~E. Gaines, J.~L. Mawhin, Coincidence degree, and nonlinear differential
  equations, vol. 568 of Lecture Notes in Mathematics, Springer-Verlag,
  Berlin-New York, 1977.

\bibitem{GaSa-1982}
R.~E. Gaines, J.~Santanilla~M., A coincidence theorem in convex sets with
  applications to periodic solutions of ordinary differential equations, Rocky
  Mountain J. Math. 12 (1982) 669--678.

\bibitem{GaHaZa-2003mod}
M.~Gaudenzi, P.~Habets, F.~Zanolin, An example of a superlinear problem with
  multiple positive solutions, Atti Sem. Mat. Fis. Univ. Modena 51 (2003)
  259--272.

\bibitem{GaHaZa-2003}
M.~Gaudenzi, P.~Habets, F.~Zanolin, Positive solutions of superlinear boundary
  value problems with singular indefinite weight, Commun. Pure Appl. Anal. 2
  (2003) 411--423.

\bibitem{GaHaZa-2004}
M.~Gaudenzi, P.~Habets, F.~Zanolin, A seven-positive-solutions theorem for a
  superlinear problem, Adv. Nonlinear Stud. 4 (2004) 149--164.

\bibitem{GiGo-2009rse}
P.~M. Gir{\~a}o, J.~M. Gomes, Multi-bump nodal solutions for an indefinite
  non-homogeneous elliptic problem, Proc. Roy. Soc. Edinburgh Sect. A 139
  (2009) 797--817.

\bibitem{GiGo-2009jde}
P.~M. Gir{\~a}o, J.~M. Gomes, Multibump nodal solutions for an indefinite
  superlinear elliptic problem, J. Differential Equations 247 (2009)
  1001--1012.

\bibitem{GoReLoGo-2000}
R.~G{\'o}mez-Re{\~n}asco, J.~L{\'o}pez-G{\'o}mez, The effect of varying
  coefficients on the dynamics of a class of superlinear indefinite
  reaction-diffusion equations, J. Differential Equations 167 (2000) 36--72.

\bibitem{GrKoWa-2008}
J.~R. Graef, L.~Kong, H.~Wang, Existence, multiplicity, and dependence on a
  parameter for a periodic boundary value problem, J. Differential Equations
  245 (2008) 1185--1197.

\bibitem{Ha-1980}
J.~K. Hale, Ordinary differential equations, 2nd ed., Robert E. Krieger
  Publishing Co., Inc., Huntington, N.Y., 1980.

\bibitem{LoSc-2012}
N.~H. Loc, K.~Schmitt, Bernstein-{N}agumo conditions and solutions to nonlinear
  differential inequalities, Nonlinear Anal. 75 (2012) 4664--4671.

\bibitem{MaReTo-2014}
A.~Margheri, C.~Rebelo, P.~J. Torres, On the use of {M}orse index and rotation
  numbers for multiplicity results of resonant {BVP}s, J. Math. Anal. Appl. 413
  (2014) 660--667.

\bibitem{Ma-1969}
J.~Mawhin, \'{E}quations int\'egrales et solutions p\'eriodiques des syst\`emes
  diff\'erentiels non lin\'eaires, Acad. Roy. Belg. Bull. Cl. Sci. (5) 55
  (1969) 934--947.

\bibitem{Ma-1972}
J.~Mawhin, Equivalence theorems for nonlinear operator equations and
  coincidence degree theory for some mappings in locally convex topological
  vector spaces, J. Differential Equations 12 (1972) 610--636.

\bibitem{Ma-1974}
J.~Mawhin, Boundary value problems for nonlinear second-order vector
  differential equations, J. Differential Equations 16 (1974) 257--269.

\bibitem{Ma-1979}
J.~Mawhin, Topological degree methods in nonlinear boundary value problems,
  vol.~40 of CBMS Regional Conference Series in Mathematics, American
  Mathematical Society, Providence, R.I., 1979.

\bibitem{Ma-1981}
J.~Mawhin, The {B}ernstein-{N}agumo problem and two-point boundary value
  problems for ordinary differential equations, in: Qualitative theory of
  differential equations, {V}ol. {I}, {II} ({S}zeged, 1979), vol.~30 of Colloq.
  Math. Soc. J\'anos Bolyai, North-Holland, Amsterdam-New York, 1981, pp.
  709--740.

\bibitem{Ma-1993}
J.~Mawhin, Topological degree and boundary value problems for nonlinear
  differential equations, in: Topological methods for ordinary differential
  equations ({M}ontecatini {T}erme, 1991), vol. 1537 of Lecture Notes in
  Mathematics, Springer, Berlin, 1993, pp. 74--142.

\bibitem{Ma-2008}
J.~Mawhin, Reduction and continuation theorems for {B}rouwer degree and
  applications to nonlinear difference equations, Opuscula Math. 28 (2008)
  541--560.

\bibitem{Nu-1973}
R.~D. Nussbaum, Periodic solutions of some nonlinear, autonomous functional
  differential equations. {II}, J. Differential Equations 14 (1973) 360--394.

\bibitem{Sc-1976}
K.~Schmitt, Fixed point and coincidence theorems with applications to nonlinear
  differential and integral equations, S{\'e}minaires de Math{\'e}matique
  Appliqu{\'e}e et M{\'e}canique, Rapport No. 97, Universit\'{e} catholique de
  Louvain, Vander, Louvain-la-Neuve, 1976.

\bibitem{Se-1976}
E.~Seneta, Regularly varying functions, vol. 508 of Lecture Notes in
  Mathematics, Springer-Verlag, Berlin-New York, 1976.

\bibitem{To-2003}
P.~J. Torres, Existence of one-signed periodic solutions of some second-order
  differential equations via a {K}rasnoselskii fixed point theorem, J.
  Differential Equations 190 (2003) 643--662.

\bibitem{Za-1987}
F.~Zanolin, On the periodic boundary value problem for forced nonlinear second
  order vector differential equations, Riv. Mat. Pura Appl. (1987) 105--124.

\end{thebibliography}

\bigskip
\begin{flushleft}

{\small{\it Preprint}}

\end{flushleft}

\end{document}